\documentclass[12pt]{article}

\usepackage{amsmath}
\usepackage{amssymb}
\usepackage{enumitem}
\usepackage{graphicx}
\usepackage{mathdots}
\usepackage{color}

\def\B{{\boldsymbol B}}
\def\bb{\boldsymbol{b}}
\def\Be{{\boldsymbol B}^{\rm even}}
\def\Bo{{\boldsymbol B}^{\rm odd}}

\def\nice{\displaystyle}

\def\End{{\rm End}}

\def\oM{\overline{\mathcal{M}}}
\def\cM{{\mathcal{M}}}

\def\Z{\mathbb{Z}}
\def\C{\mathbb{C}}
\def\Q{\mathbb{Q}}
\def\d{\partial}
\def\qed{{\hfill $\Diamond$}}
\def\b1{{\mathbf 1}}

\def\hd{{\widehat \partial}}
\def\hhd{{\widetilde{\partial}}}

\def\Aut{{\rm Aut}}
\def\F{{\mathsf{F}}}

\def\HH{\mathcal{H}}

\def\W{\mathsf{W}}
\def\D{\mathsf{D}}

\def\slc{{\mathsf{sl}}_2(\mathbb{C})}
\def\sl{{\mathsf{sl}}}

\def\blangle{\big\langle}
\def\brangle{\big\rangle}
\def\Blangle{\Big\langle}
\def\Brangle{\Big\rangle}

\usepackage{amsthm}

\newtheorem{theorem}{Theorem}
\newtheorem{proposition}{Proposition}[section]

\newtheorem{lemma}[proposition]{Lemma}
\newtheorem{conjecture}[proposition]{Conjecture}
\theoremstyle{definition}
\newtheorem{definition}[proposition]{Definition}
\newtheorem{example}[proposition]{Example}
\newtheorem{remark}[proposition]{Remark}

\title{Tautological relations via $r$-spin structures}
\author{R. Pandharipande, A. Pixton, D. Zvonkine}
\date{February 2019}

\begin{document}

\maketitle

\begin{abstract}
Relations among tautological classes on $\oM_{g,n}$
are obtained via the study of Witten's $r$-spin theory
for higher $r$. In order to calculate the quantum product,
a new formula relating the $r$-spin correlators in genus 0
to the representation theory of $\slc$ is proven.
The Givental-Teleman classification of CohFTs is used
at two special semisimple points of the associated
Frobenius manifold. At the first semisimple point, the
$R$-matrix is exactly solved in terms of hypergeometric
series. As a result, 
an explicit formula for Witten's $r$-spin class is obtained (along
with tautological relations in higher degrees). As an application,
the $r=4$ relations are used to bound the Betti numbers of
$R^*(\cM_g)$. At the second semisimple point, the form of
the $R$-matrix 
implies a polynomiality property in $r$ of Witten's $r$-spin class.

In the Appendix (with F. Janda), a conjecture relating the\linebreak $r=0$ limit of 
Witten's $r$-spin class to the class of the moduli space
of holomorphic differentials is presented.
\end{abstract}

\setcounter{tocdepth}{1} 

\baselineskip=8pt
\tableofcontents

\baselineskip=15pt



\pagebreak
\setcounter{section}{-1}
\section{Introduction}

\subsection{Overview} \label{s000}
Let $\oM_{g,n}$ be the moduli space of stable genus $g$ curves with $n$ markings. Let
$$RH^*(\oM_{g,n}) \subset H^*(\oM_{g,n})$$
be the subring of tautological 
classes in cohomology{\footnote{All cohomology and Chow groups
will be taken with $\mathbb{Q}$-coefficients.
The tautological ring in Chow is denoted by $$R^*(\oM_{g,n})\subset A^*(\oM_{g,n})\, .$$
We will use the complex grading for $RH^*(\oM_{g,n})$, so
$R^i(\oM_{g,n}) \rightarrow RH^i(\oM_{g,n})$.}}.
The subrings  $$\Big\{\, RH^*(\oM_{g,n}) \subset H^*(\oM_{g,n})\, \Big\}_{g,n}$$ 
are defined together as {\em the smallest
system of $\mathbb{Q}$-subalgebras closed under push-forward via all
boundary and forgetful maps}, see \cite{FP1,FP3,GraPan}.
There has been substantial progress in the understanding of      
 $RH^*(\oM_{g,n})$ since the study began in the 1980s \cite{Mum}. 
The subject took  a new turn in 2012 with the
family of relations conjectured in \cite{Pixton}. We refer the
reader to \cite{SLC} for a survey of recent developments.

Witten's $r$-spin class defines a Cohomological Field Theory (CohFT) for each integer $r\geq 2$.
Witten's $2$-spin theory concerns only  
the fundamental classes of the moduli spaces
of curves (and leads to no new geometry).
In our previous paper \cite{PaPiZv}, we used Witten's 3-spin 
theory to construct a family of relations among tautological classes of $\oM_{g,n}$ equivalent (in cohomology) to the relations
proposed in \cite{Pixton}.
Our goal here is to extend our study of tautological relations 
to Witten's $r$-spin theory for all $r \geq 3$.

Taking \cite{PaPiZv} as a starting point, Janda has completed
 a formal study of 
tautological relations obtained from CohFTs.
Two results of Janda  are directly relevant
here:
\begin{enumerate}
\item[(i)] The relations for $r= 3$ are valid in Chow 
\cite{Janda14, Janda15*}.
\item[(ii)] The relations for $r\geq 4$ are implied by the relations 
for 
$r=3$ \cite{Janda14,Janda15}.
\end{enumerate}
By (i) and (ii) together, all of the $r$-spin relations that we find will be valid in Chow.
However, since our methods here are cohomological, we will use the
language of cohomology throughout the paper.

\pagebreak
Given Janda's results, why proceed with the higher $r$-spin analysis?
There are three basic reasons:
\begin{enumerate}
\item[$\bullet$] 
The $r=4$ relations are simpler and easier to use than the 3-spin
relations when restricted to $\cM_g$. From the 4-spin relations, 
we derive a new bound
on the rank of $R^d(\cM_g)$ which specializes in case $d\geq g-2$ to 
the basic results
\begin{equation}
\text{dim}_{\mathbb{Q}}\, R^{g-2}(\cM_g) \leq 1\, , \ \ \ \ \
\text{dim}_{\mathbb{Q}}\, R^{>g-2}(\cM_g)=0 \label{ftft}
\end{equation}
of Looijenga \cite{Looij}. By (ii)  above, we 
conclude that both statements in \eqref{ftft} follow from the restriction of the 3-spin relations to $\cM_g$.
The latter restriction equals the Faber-Zagier relations{\footnote{See \cite{SLC} for a
survey of the Faber-Zagier relations and related topics.}}. 

The outcome is
a proof that Looijenga's results 
\eqref{ftft} follow from the Faber-Zagier 
relations. Since Faber's conjectures \cite{Faber} governing the 
proportionalities of $\kappa$ monomials in $R^{g-2}(\cM_g)$
are known{\footnote{A proof (unpublished) was found
by Faber and Zagier in 2002. The result
is also derived in \cite[Section 3]{PiThesis}.}} 
to be compatible with the Faber-Zagier relations,
we can also conclude that the Faber-Zagier relations
{\em imply} these proportionalities.

Relations, by themselves, cannot prove non-vanishing results.
The nonvanishing
\begin{equation}\label{ltt2}
\text{dim}_{\mathbb{Q}}\, R^{g-2}(\cM_g)\geq 1
\end{equation}
is proven \cite{Faber,FP1} by Hodge integral evaluations. The results
\eqref{ftft} and \eqref{ltt2} together prove 
$$R^{g-2}(\cM_g)\cong\mathbb{Q}\, .$$

\item[$\bullet$] Another outcome is a much better understanding of Witten's $r$-spin
class for higher $r$. We obtain an exact formula for Witten's $r$-spin
correlators in genus 0
in terms of the
representation theory of ${\mathsf{sl}}_2(\mathbb{C})$. 
The genus 0 results and the Givental-Teleman classification of semisimple CohFTs together
provide two explicit approaches to Witten's $r$-spin class in all genera.
The first leads to a formula for all $r$ parallel to the $3$-spin
formula of \cite{PaPiZv}. 
As an application of the second approach, we prove a new {\em polynomiality} property in $r$
of Witten's $r$-spin
class.

\item[$\bullet$] In the Appendix with Janda, we present a new 
conjecture relating
an appropriate limit (defined by polynomiality) of Witten's $r$-spin class to 
the class of the moduli space of holomorphic differentials with prescribed
zero multiplicities.

\end{enumerate}

\noindent These are unexpected developments. There
is no {\em a priori} reason to believe the $4$-spin relations
would be algebraically simpler or that polynomiality in $r$
holds for Witten's class. The very simple connection with the class
of the moduli
space of differentials leads to a much more direct
calculation than in \cite[Appendix]{FarP},  but is available
only in the holomorphic case.

\subsection{Cohomological field theories} \label{cft}
We recall here the basic definitions of a cohomological field theory by Kontsevich and Manin~\cite{KonMan}, see also \cite{PPP} for a survey.

Let $V$ be a finite dimensional $\mathbb{Q}$-vector space with 
a non-degenerate symmetric 2-form $\eta$ and a 
distinguished element $\b1 \in V$.
The data $(V,\eta, \b1)$ is the starting point for defining 
a cohomological field theory.
Given a 
basis $\{e_i\}$ of $V$, we write the 
symmetric form as a matrix
$$\eta_{jk}=\eta(e_j,e_k) \ .$$ The inverse matrix is denoted by $\eta^{jk}$ as usual.

A {cohomological field theory} consists of 
a system $\Omega = (\Omega_{g,n})_{2g-2+n > 0}$ of elements 
$$
\Omega_{g,n} \in H^*(\oM_{g,n}) \otimes (V^*)^{\otimes n}.
$$
We view $\Omega_{g,n}$ as associating a cohomology class on $\oM_{g,n}$ 
to elements of $V$ assigned to the $n$ markings.
The CohFT axioms imposed on $\Omega$ are:

\begin{enumerate}
\item[(i)] Each $\Omega_{g,n}$ is $S_n$-invariant, where the action of the symmetric group $S_n$ permutes both the marked points of $\oM_{g,n}$ and the
copies of $V^*$.
\item[(ii)] Denote the basic gluing maps by
$$
q : \oM_{g-1, n+2} \to \oM_{g,n}\ ,
$$
$$
r: \oM_{g_1, n_1+1} \times \oM_{g_2, n_2+1} \to \oM_{g,n}\ .
$$
The pull-backs $q^*(\Omega_{g,n})$ and $r^*(\Omega_{g,n})$ are
 equal to the contractions of $\Omega_{g-1,n+2}$ and 
$\Omega_{g_1, n_1+1} \otimes \Omega_{g_2, n_2+1}$ by the bi-vector 
$$\sum_{j,k} \eta^{jk} e_j \otimes e_k$$ inserted at the two identified points. 
\item[(iii)] Let $v_1, \dots, v_n \in V$ be any vectors, and let $p: \oM_{g,n+1} \to \oM_{g,n}$ be the forgetful map. We require 
$$
\Omega_{g,n+1}(v_1 \otimes \cdots \otimes v_n \otimes \b1) = p^*\Omega_{g,n} (v_1 \otimes \cdots \otimes v_n)\ ,
$$
$$\Omega_{0,3}(v_1\otimes v_2 \otimes \b1) = \eta(v_1,v_2)\ .$$
\end{enumerate}

\begin{definition}\label{defcohft}
A system $\Omega= (\Omega_{g,n})_{2g-2+n>0}$ of elements 
$$
\Omega_{g,n} \in H^*(\oM_{g,n}) \otimes (V^*)^{\otimes n}
$$
satisfying properties~(i) and (ii) is a {\em cohomological field theory} or a {\em CohFT}. If (iii) is also satisfied, $\Omega$ is 
a {\em CohFT with unit}.
\end{definition}

A CohFT $\Omega$ yields a {\em quantum product} $\bullet$ on $V$ via 
$$\eta(v_1 \bullet v_2, v_3) = \Omega_{0,3}(v_1 \otimes v_2 \otimes v_3)\ .$$
Associativity of $\bullet$ follows from (ii). The element
$\b1\in V$ is the identity for $\bullet$ by (iii).

A CohFT $\omega$ composed only of degree~0 classes,
$$\omega_{g,n} \in H^0(\oM_{g,n}) \otimes (V^*)^{\otimes n}\ ,$$
 is called a {\em topological field theory}. 
Via property (ii), $\omega_{g,n}(v_1, \dots, v_n)$ 
is determined by considering stable curves with a maximal number of nodes. 
Such a curve is obtained by identifying several rational curves with three marked points. The value of $\omega_{g,n}(v_1 \otimes \cdots \otimes v_n)$ is 
thus uniquely specified by the values of $\omega_{0,3}$ and by the quadratic form~$\eta$. In other words, given $V$ and $\eta$, a topological field theory is uniquely determined by the associated quantum product.

\subsection{Witten's $r$-spin class} \label{wsc}
For every integer $r \geq 2$, there is a beautiful CohFT obtained from 
Witten's $r$-spin class. We review here the basic properties of
the construction. The integer $r$ is fixed once and for all.
  
Let $V_r$ be an $(r-1)$-dimensional $\mathbb{Q}$-vector space with basis
 $e_0, \dots, e_{r-2}$, bilinear form 
$$
\eta_{ab} = \eta(e_a, e_b) =\delta_{a+b,r-2} \, ,
$$
and unit vector $\b1 = e_0$. Witten's $r$-spin theory provides a family of classes
$$
W^r_{g,n}(a_1, \dots, a_n) \in H^*(\oM_{g,n}).
$$ 
for $a_1, \dots, a_n \in \{0, \dots, r-2 \}$.
These define a CohFT by
$$
\W^r_{g,n}: V^{\otimes n} \rightarrow H^*(\oM_{g,n}), 
\ \ \ 
\W^r_{g,n}( e_{a_1} \otimes \cdots \otimes e_{a_n}) =
W^r_{g,n}(a_1, \dots, a_n)\ .
$$
Witten's class $W^r_{g,n}(a_1, \dots, a_n)$ has (complex) degree given
by the  formula
\begin{eqnarray}
\label{gred}
\text{deg}_{\C}\ W^r_{g,n}(a_1, \dots, a_n) & = & 
\D^r_{g,n}(a_1, \dots, a_n) \\ \nonumber
& = & \frac{(r-2)(g-1) + \sum_{i=1}^n a_i}{r}\ .
\end{eqnarray}
If $\D^r_{g,n}(a_1, \dots, a_n)$ is not an integer, the corresponding
Witten's class vanishes.

In genus 0, the construction was first carried out by Witten \cite{Witten}
using $r$-spin structures ($r^{\rm th}$ roots of the canonical bundle)
and satisfies the following initial conditions:
\begin{equation}\label{fred}
W^r_{0,3}(a_1,a_2,a_3) = 
\left|
\begin{array}{cl}
1 & \mbox{ if } a_1+a_2+a_3 = r-2,\\
0 & \mbox{ otherwise.}
\end{array}
\right.
\end{equation}
$$
W^r_{0,4}(1,1,r-2,r-2) = \frac1{r}\cdot [{\mathrm{pt}}] \ \in H^2(\overline{M}_{0,4})\ .
$$
Uniqueness of Witten's $r$-spin theory in
genus~0 follows  from the initial conditions \eqref{fred}
and the axioms of a CohFT with unit.

The genus $0$ sector defines a quantum product $\bullet$
on $V$ with unit $e_0$,
$$ \eta(e_a\bullet e_b, e_c) = W^r_{0,3}(a,b,c) \ .$$
The resulting algebra, even after
extension to $\C$, is not semisimple.

The existence of Witten's class in higher genus is both remarkable and
highly non-trivial. 
An algebraic construction was first obtained by Polishchuk and Vaintrob~\cite{PolVai} defining 
$$W^r_{g,n}(a_1, \dots, a_n)\in A^*(\oM_{g,n})$$
as a cycle class. The algebraic approach was 
later simplified by Chiodo~\cite{Chiodo}. 
Analytic constructions have been given by Mochizuki ~\cite{Mochizuki} and
later by Fan, Jarvis, and Ruan \cite{fjr}.
As a consequence of the following result, the analytic
and algebraic approaches coincide and yield tautological classes
in cohomology.

\begin{theorem}[\cite{PaPiZv}] \label{Thm:Wittensclass}
For every $r\geq 2$, 
there is a unique CohFT which extends Witten's $r$-spin theory
in genus 0 and has pure dimension \eqref{gred}.
The unique extension takes values in the tautological ring
$$RH^*(\oM_{g,n})\subset H^*(\oM_{g,n}).$$
\end{theorem}

\vspace{8pt}
Whether Witten's $r$-spin theory as an algebraic cycle
takes values in
$$R^*(\oM_{g,n})\subset A^*(\oM_{g,n})$$
 is an open question.

\subsection{Witten's $r$-spin class and representations of $\slc$.}
Consider the Lie algebra $\sl_2=\slc$.
Denote by $\rho_k$ the $k$-th symmetric power of the standard 2-dimensional representation of $\sl_2$,
$$\rho_k =\text{Sym}^k(\rho_1)\, , \ \ \   \dim \rho_k = k+1\, .$$
The complete list
of irreducible representations of $\sl_2$ is
 $\{ \rho_k\}_{k \geq 0}$, see \cite{FH}.
Let
$$
H = 
\frac12 
\begin{pmatrix}
1 & 0 \\
0 & -1
\end{pmatrix}
\in \sl_2\, .
$$
The trace of the exponential is  
$$
{\rm tr}_k e^{tH} = \frac{e^{(k+1)t/2} - e^{-(k+1)t/2}}{e^{t/2} - e^{-t/2}}\, .
$$
The formula for the tensor product of two irreducible representations
is then easily obtained: 
$$
\rho_k \otimes \rho_l = 
\rho_{|k-l|} \oplus \rho_{|k-l|+2} \oplus \cdots \oplus \rho_{k+l}\, .
$$
Our first result relates Witten's $r$-spin class in genus 0 with
the representation theory of $\sl_2$. 

\begin{theorem} \label{Thm:sl2}
Let  ${\bf{a}}=(a_1, \dots, a_{n\geq 3})$ with 
$a_i \in \{ 0, \dots, r-2 \}$  satisfy the degree constraint
$\D^r_{0,n}({\bf{a}})= n-3$.
Then, 
\begin{multline*}
W^r_{0,n}({\bf{a}}) = 
\frac{(n-3)!}{r^{n-3}} 
\dim \Bigl[ \rho_{r-2-a_1} \otimes \cdots \otimes \rho_{r-2-a_n} \Bigr]^{\sl_2} \cdot [{\rm pt}]\,   \in H^{2(n-3)}(\oM_{0,n}),
\end{multline*}
where the superscript $\sl_2$ denotes the $\sl_2$-invariant subspace and
the class 
$$[{\rm pt}]\in H^{2(n-3)}(\oM_{0,n})$$ is  Poincar\'e dual to a point.
\end{theorem}

\vspace{5pt}
The degree constraint $\D^r_{0,n}({\bf{a}})= n-3$ in the statement of
 Theorem \ref{Thm:sl2} can be written
 equivalently (using \eqref{gred}) as
$$\sum_{i=1}^n a_i = (n-2)r-2\, .$$
Since $a_i\leq r-2$, the bound $n\leq r+1$ is a simple consequence.

\subsection{Shifted Witten class}
\label{frrz}
Given a vector $\gamma \in V_r$, the {\em shifted Witten class} is defined by 
$$
\W_{g,n}^{r,\gamma}(v_1 \otimes \cdots \otimes v_n) = 
\sum_{m \geq 0}\frac{1}{m!}\,
p_{m*}\, \W^r_{g,n+m}(v_1 \otimes \cdots \otimes v_n \otimes \gamma^{\otimes m}),
$$
where $p_m \colon \oM_{g,n+m} \to \oM_{g,n}$ is the forgetful map.
The shifted Witten class $\W^{r,\gamma}$ determines a CohFT, see
\cite[Section 1.1]{PaPiZv}.

The vector space $V_r$ carries a {\em Gromov-Witten potential} $\F$ satisfying 
$$
\frac{\d^3 \F}{\d t^a \d t^b \d t^c}(\gamma) = 
\W^{r,\gamma}_{0,3}(e_a \otimes e_b \otimes e_c)
$$
which defines a Frobenius manifold structure on $V_r$. 

\begin{example}{\em
For $r=3$, the Gromov-Witten potential obtained from  Witten's class equals
$$
\F(x,y) = \frac12 x^2 y +\frac1{72} y^4,
$$
where $x = t^0$ and $y = t^1$.

For $r=4$, the potential is 
$$
\F(x,y,z) = \frac12 x^2 z + \frac12 xy^2 + \frac1{16} y^2z^2 + \frac1{960} z^5,
$$
where $x = t^0$, $y = t^1$, and  $z= t^2$.}
\end{example}

The tangent vector space to $\gamma \in V_r$ has a natural Frobenius (or fusion) algebra structure $\Phi^{r,\gamma}$ given by the structure constants
$$
 \eta( e_a\bullet_\gamma e_b, e_c) =
\frac{\d^3 \F}{\d t^a \d t^b \d t^c}(\gamma).
$$

\begin{theorem} \label{Thm:Verlinde}
For $(0, \ldots, 0, r) \in V_r$,  the algebra 
$\Phi^{r,(0,\ldots,0,r)}$ is isomorphic to the Verlinde algebra of level~$r$ for $\sl_2$. 
\end{theorem}

Since the Verlinde algebra is semisimple,
 the algebra
$\Phi^{r,(0,\ldots,0,r)}$
is semisimple as a consequence of Theorem \ref{Thm:Verlinde}.
Proposition \ref{Prop:idempotents} of Section \ref{Ssec:TopFT1}
provides a  basis of idempotents for the algebras
$\Phi^{(0,\ldots,0,r\phi)}$ for all $0\neq \phi\in {\mathbb{Q}}$.

Theorems~\ref{Thm:sl2} and~\ref{Thm:Verlinde} relate  Witten's $r$-spin class and the corresponding Frobenius manifold to the representations of $\sl_2$ at level~$r$. On the other hand, the Frobenius manifold associated to Witten's
$r$-spin theory is usually constructed from the $A_{r-1}$ singularity and thus is related to the Lie algebra $\sl_r$ Lie algebra. 
Perhaps there is some form of rank-level duality for Frobenius manifolds, 
but we are not aware of other examples.

\subsection{Euler field and Hodge grading operator.}
The Frobenius manifold structure on~$V_r$ includes an Euler field and a 
conformal dimension which determine a Hodge grading operator. 

The {\em Euler field} on the Frobenius manifold~$V_r$ is
$$
E = \sum_{a=0}^{r-2} \left(1 - \frac{a}r\right) t^a \d_a\, .
$$
The Lie derivatives with respect to $E$ of the basis vectors fields are easily
calculated:
$$
L_E(\d_a) = [E, \d_a] = -\left(1 - \frac{a}r\right) \d_a\, .
$$
The {\em conformal dimension} equals 
$$
\delta = \frac{r-2}{r}\, .
$$

Let $v$ be a tangent vector at a point of the Frobenius manifold. We define the {\em shifted degree operator} $\mu(v)$, also called the {\em Hodge grading operator}, by
$$
\mu(v) = [E,v] + \left(1-\frac{\delta}{2}\right)v\ .
$$
Here, the vector~$v$ is extended to a flat tangent vector field in 
order to compute the commutator.
We  have
$$
\mu(\d_a) = \frac{2a+2-r}{2r} \d_a\, .
$$

\subsection{Tautological relations} \label{trtr}
We will construct tautological relations using Givental's $R$-matrix action
on Coh\-FTs.{\footnote{A review of Givental's $R$-matrix
action  on CohFTS
in the form we require can be found in  \cite[Section 2]{PaPiZv}.}} 
The relations will be proven by studying Witten's $r$-spin class.
The point
$$\tau=(0,\ldots,0,r\phi)\in V_r\, $$
with respect to
 a nonzero parameter $\phi \in \mathbb{Q}$ will play a special role. Let
$$
\hd_a = \phi^{-(2a-r+2)/4} \d_a
$$
be a new tangent frame on the Frobenius manifold $V_r$ at the point
$\tau$.

 We define a multilinear
map 
$$\omega^{r,\tau}_{g,n}: V_r^{\otimes n} \rightarrow \mathbb{Q}$$
by the trigonometric formula
\begin{multline}\label{f99}
\omega^{r,\tau}_{g,n}\left(\hd_{a_1} \otimes \cdots \otimes \hd_{a_n}\right) 
= \\
\left(\frac{r}2\right)^{g-1}
\phi^{(r-2)(2g-2+n)/4} \; 
\sum_{k=1}^{r-1} 
\frac
{(-1)^{(k-1)(g-1)} \prod\limits_{i=1}^n \sin \left(
\frac{(a_i+1)k\pi}r \right)}
{\left(\sin(\frac{k\pi}{r})\right)^{2g-2+n}}\, .
\end{multline}
We will prove  $\omega^{r,\tau}$ is a CohFT with
the right side of \eqref{f99} interpreted as a multiple
of the identity $1\in H^0(\oM_{g,n})$. In fact, 
$\omega^{r,\tau}$ is the topological part of the $\tau$-shifted
$r$-spin
CohFT $\mathsf{W}^{r,\tau}$ defined in Section \ref{frrz}.

Our construction of tautological relations
depends upon the following hypergeometric
series.
For every $a \in \{0, \dots, r-2 \}$, we define 
$$
\B_{r,a}(T) = \sum_{m=0}^\infty
\left[
\prod_{i=1}^m 
\frac{\big((2i-1)r-2(a+1)\big)\big((2i-1)r+2(a+1)\big)}i 
\right] \!\!
\left( - \frac{T}{16r^2} \right)^{\! m} \!\!\! .
$$
For $r$ even and $a = \frac{r}2-1$, we have $\B_{r,a} = 1$. 
Otherwise, $\B_{r,a}$ is a power series
with all coefficients nonzero. We denote by $\Be_{r,a}$ and $\Bo_{r,a}$ the even and odd parts of the power series $\B_{r,a}$,
$$\B_{r,a}(T)= \Be_{r,a}(T)+\Bo_{r,a}(T)\, .$$

\begin{example}{\em
For $r=3$, we obtain a slight variation of the
series occurring in the Faber-Zagier relations:
\begin{align*}
\B_{3,0}(T) &= \sum_{m \geq 0} \frac{(6m)!}{(2m)! \, (3m)!} \left(- \frac{T}{1728} \right)^m\, ,\\
\B_{3,1}(T) &= \sum_{m \geq 0} \frac{1+6m}{1-6m} \frac{(6m)!}{(2m)! \, (3m)!} \left(- \frac{T}{1728} \right)^m\, .
\end{align*}
For $r=4$, we obtain:
\begin{align*}
\B_{4,0}(T) &= \sum_{m \geq 0} \frac{(4m)!}{m! \, (2m)!} \left(- \frac{T}{256} \right)^m,\\
\B_{4,1}(T) &= 1\, ,\\
\B_{4,2}(T) &= \sum_{m \geq 0} \frac{1+4m}{1-4m} \frac{(4m)!}{m! \, (2m)!} \left(- \frac{T}{256} \right)^m\, .
\end{align*}}
\end{example}

Consider the matrix-valued power series $R(z) \in \End(V_r)[[z]]$ with coefficients given by
$$
R^a_a = \Be_{r,r-2-a}(\phi^{-r/2}z)\, , 
\ \ \ \ a \in \{0, \dots, r-2 \}
$$
on the main diagonal,
$$
R^{r-2-a}_a =  -\Bo_{r,a}(\phi^{-r/2}z)\, , 
\ \ \ \ a \in \{0, \dots, r-2 \}
$$
on the antidiagonal (if $r$ is even, the coefficient at the intersection of both diagonals is 1), and $0$ everywhere else.

\begin{example}{\em
For $r=3$, the $R$ matrix is
$$
\begin{pmatrix}
\ \ \ \Be_{3,1}(\phi^{-3/2}z) & -\Bo_{3,1}(\phi^{-3/2}z) \\
-\Bo_{3,0}(\phi^{-3/2}z) &\ \ \ \Be_{3,0}(\phi^{-3/2}z)
\end{pmatrix}\, . $$
For $r=4$, the $R$ matrix is
$$
\begin{pmatrix}
\ \ \ \Be_{4,2}(\phi^{-2}z) & \ 0\  & -\Bo_{4,2}(\phi^{-2}z)   \\
0 & \ 1\  & 0 \\
-\Bo_{4,0}(\phi^{-2}z) &\ 0\  & \ \,  \Be_{4,0}(\phi^{-2}z)
\end{pmatrix}\, . $$}
\end{example}

\vspace{5pt}
We will prove that the inverse matrix $R^{-1}(z)$ has coefficients
$$
(R^{-1})^a_a = \Be_{r,a}(\phi^{-r/2}z)\, , 
\ \ \ \ a \in \{0, \dots, r-2 \}
$$
on the main diagonal,
$$
(R^{-1})^{r-2-a}_a = \Bo_{r,a}(\phi^{-r/2}z)\, , 
\ \ \ \ a \in \{0, \dots, r-2 \}
$$
on the anti-diagonal (if $r$ is even, the coefficient at the intersection of both diagonals is 1), and $0$ everywhere else.

Let $\Omega^{r,\tau}$ be the stable graph expression for the CohFT obtained by the action  $\omega^{r,\tau}$ of the above $R$-matrix,
$$\Omega^{r,\tau}= R.\omega^{r,\tau},$$
see \cite[Definition 2.13]{PaPiZv}.

\begin{theorem} \label{Thm:TautRel1}
For every $d > \D^r_{g,n}(a_1, \dots, a_n)$, the degree $d$ part of
$$\Omega^{r,\tau}_{g,n}(e_{a_1} \otimes \cdots \otimes e_{a_n})\in H^*(\oM_{g,n})$$ 
vanishes.
\end{theorem}

The complexity of the topological field theory $\omega^{r,\tau}$ 
for higher $r$ leads to 
complicated relations on $\oM_{g,n}$. However, by multiplying by $\psi$ 
classes, pushing forward by forgetful maps, and then restricting to 
the interior, we obtain much simpler relations on $\cM_{g,n}$.
In order to write the resulting relations,
we  extend the definition of the power series $\B_{r,a}(T)$ to all 
$$a\ge 0 \ \ \ {\text{satisfying}}\ \ \  a\not\equiv r-1 \ {\text{ mod }} r$$
 by the formula
\[
\B_{r,a+rb}(T) = T^b\B_{r,a}(T)\, .
\]

The relations depend upon on a partition{\footnote{A partition
has positive integer parts. The empty partition is permitted.}} 
$$\sigma=(\sigma_1,\sigma_2,\ldots,\sigma_{\ell(\sigma)})$$ with
{\em no} part $\sigma_i$ congruent to $r-1$ mod $r$, and a vector
of non-negative integers $${\bf{a}}=(a_1,\ldots,a_n)$$ with
{\em no} $a_i$ congruent to $r-1$ mod $r$.

\begin{theorem}\label{Cor:Mg}
Let $\sigma$ and $\bf{a}$ avoid $r-1$ mod $r$, and let $d$
satisfy 
\begin{equation*}
rd > (r-2)(g-1) + |\sigma| + \sum_{i=1}^n a_i
\end{equation*}
and the parity condition
\begin{equation*}
rd \equiv (r-2)(g-1) + |\sigma| + \sum_{i=1}^n a_i \, \mod{2}.
\end{equation*}
Then, the degree $d$ part of 
\[
\prod_{i=1}^n \B_{r,a_i}(\psi_i)
\left(\sum_{m\ge 0}\frac{1}{m!} p_{\ell(\sigma)+m*}\prod_{j=1}^{\ell(\sigma)}\B_{r,\sigma_j+r}(\psi_{n+j})\prod_{k=1}^m(T-T\B_{r,0})(\psi_{n+\ell(\sigma)+k})\right)\]
vanishes in  $H^{2d}(\cM_{g,n})$.
\end{theorem}

In the statement of Theorem \ref{Cor:Mg}, $p_{\ell(\sigma)+m}$   is  (as before) the forgetful map
\[
p_{\ell(\sigma)+m}:\oM_{g,n+\ell(\sigma)+m}\to\oM_{g,n}
\]
forgetting the last $\ell(\sigma)+m$ points.



We use the relations of Theorem \ref{Cor:Mg} to bound the Betti numbers of the tautological ring of $\cM_g$. Let $P(n,k)$ denote the set of partitions of $n$ of length at most $k$.
\begin{theorem}\label{Thm:betti}
For $g \ge 2$ and $d \ge 0$,
$
\ \dim_\Q RH^d(\cM_g) \le |P(d,g-1-d)|.
$
\end{theorem}
The bound of Theorem \ref{Thm:betti} implies the results
\begin{equation}
\text{dim}_{\mathbb{Q}}\, RH^{g-2}(\cM_g) \leq 1\, , \ \ \ \ \
\text{dim}_{\mathbb{Q}}\, RH^{>g-2}(\cM_g)=0 \label{ftftft} \, .
\end{equation}
For $d < g-2$,  Theorem \ref{Thm:betti} is new (no
non-trivial bounds were known before),
but is not expected to be sharp. For example,
Theorem \ref{Thm:betti} yields
$$\text{dim}_{\mathbb{Q}}\, RH^{g-3}(\cM_g) \leq 1+\lfloor\frac{g-3}{2}\rfloor\, ,$$
while the expectation based on calculations is
$$\text{dim}_{\mathbb{Q}}\, RH^{g-3}(\cM_g) =1\, .$$

Though our proofs use cohomological
methods, results (i) and (ii) of Janda discussed in Section \ref{s000}
imply Theorems \ref{Thm:TautRel1}-\ref{Thm:betti} all are valid in Chow.
In particular, Theorem \ref{Thm:betti} yields the bound
$$
\ \dim_\Q R^d(\cM_g) \le |P(d,g-1-d)|\, ,
$$
and \eqref{ftftft} specializes to Looijenga's result \eqref{ftft}.

\subsection{Polynomiality}\label{polly}


Let $a_1, \dots, a_n$ be non-negative integers satisfying the condition
$$\sum_{i=1}^n a_i = 2g-2\, .$$
If $a_i\leq r-2$ for all $i$, 
then Witten's $r$-spin class $W^r_{g,n}(a_1, \dots, a_n)$ is well-defined
and of degree
{\em independent} of the choice of $r$,
$$\D^r_{g,n}(a_,\ldots,a_n) = \frac{(r-2)(g-1) + \sum_{i=1}^n a_i}{r} =g-1\, .$$
We may reasonably ask{\footnote{The question was posed to
us by P. Rossi.}} here about the dependence of
$W^r_{g,n}(a_1, \dots, a_n)$ on $r$.

\begin{theorem}\label{Thm:poly}
For $\sum_{i=1}^n a_i = 2g-2$,
$$
r^{g-1}\cdot W^r_{g,n}(a_1, \dots, a_n)\in RH^{g-1}(\oM_{g,n})
$$
is a polynomial in $r$ for all sufficiently large $r$.
\end{theorem}

Our proof of Theorem \ref{Thm:poly} is obtained by the analysis
of the shifted CohFT $\mathsf{W}^{r,\widetilde{\tau}}$
at the point
$$\widetilde{\tau}=(0,r\phi,0,\ldots, 0)\in V_r\,. $$
The CohFTs
$\mathsf{W}^{r,{\tau}}$ and 
$\mathsf{W}^{r,\widetilde{\tau}}$ behave differently and
yield different insights.
While our knowledge of the $R$-matrix at $\widetilde{\tau}$ is {\em not}
as explicit as at $\tau$, the property of polynomiality is 
easier to see at $\widetilde{\tau}$. 

In the Appendix (with F. Janda), we conjecture
the constant term of the polynomial of Theorem~\ref{Thm:poly} is
$(-1)^g$ times the class of the closure of the locus of 
holomorphic differentials with zero multiplicities given
by $(a_1,\ldots,a_n)$. In addition to a precise formulation,
the evidence for the conjecture and the connection
to the conjectures of \cite[Appendix]{FarP} are discussed in
the Appendix.

\subsection{Plan of the paper}
We start in Section \ref{Sec:sl2}
 with the proof of Theorem \ref{Thm:sl2}. The result plays a basic role in our
analysis of the $r$-spin theory in genus $0$. The study of
the shift  
$$\tau=(0,\ldots,0, r\phi)\in V_r$$
is presented in Section \ref{Sec:lastshift}.
Theorem \ref{Thm:Verlinde} is proven in Section \ref{d11d},
and the corresponding $R$-matrix is solved in terms of hypergeometric
series in Section \ref{Ssec:Rmatrix1}.
Section \ref{ttttt} concerns the tautological relations obtained
from the shifted CohFT $\mathsf{W}^{r,{\tau}}$. Theorems 
\ref{Thm:TautRel1}, \ref{Cor:Mg}, and \ref{Thm:betti} are proven in Sections
\ref{dddd4}, \ref{Subsec:Mg}, and \ref{Subsec:betti} respectively.
The study of
the shift  
$$\widetilde{\tau}=(0,r\phi,0,\ldots,0)\in V_r$$
is presented in Section \ref{pfpfpf}.
The polynomiality of the $R$-matrix is derived in Section \ref{Ssec:Rmatrix2},
and Theorem \ref{Thm:poly} is proven in Section \ref{polypoly}.

Two formulas for Witten's $r$-spin class are given:
Theorem \ref{Thm:rWitten} of Section \ref{dddd4} via the
$\tau$-shift and  Theorem \ref{Thm:SecondShift} of Section \ref{Ssec:Rmatrix2} via the $\widetilde{\tau}$-shift. The Appendix (with F. Janda) conjectures a connection
between 
Witten's $r$-spin class and the class of the moduli space
of holomorphic differentials.

\subsection{Acknowledgments}

We are grateful to A.~Buryak,
A.~Chiodo, E.~Clader, C.~Faber, G.~Farkas,
 J.~Gu\'er\'e, F.~Janda, Y.-P.~Lee, P.~Rossi, Y.~Ruan, and S.~Shadrin for 
 discussions about
Witten's $r$-spin theory. 
The {\em Workshop on Pixton's conjectures} at
ETH Z\"urich  in October 2014 played in an important
role in our collaboration. A.~P. and D.~Z. have been
frequent guests of the 
Forschungsinstitut f\"ur Mathematik (FIM) at ETHZ.
The conjecture of the Appendix was formulated at the conference {\em Moduli spaces of
holomorphic differentials} at Humboldt University in  Berlin 
in February 2016.

R.~P. was partially supported by 
SNF-200021143\-274, SNF-200020162928, ERC-2012-AdG-320368-MCSK,
ERC-2017-AdG-786580-MACI,
SwissMAP, and
the Einstein Stiftung. 
A.~P. was supported by a fellowship from the Clay Mathematics Institute.
D.~Z. was supported by the grant ANR-09-JCJC-0104-01.
This project has received funding from the European Research
  council (ERC) under the European Union Horizon 2020 research and
  innovation program (grant agreement No. 786580).

\section{Representations of $\sl_2(\mathbb{C})$}
\label{Sec:sl2}
\subsection{Correlators}
Our goal here is to prove Theorem~\ref{Thm:sl2} relating Witten's class 
and the representation theory of $\sl_2=\sl_2(\mathbb{C})$.

\begin{definition}
Let $a_1,\ldots,a_n \in \{ 0, \dots, r-2 \}$ satisfy
 $$\sum_{i=1}^n a_i = (n-2)r - 2\, .$$
 The associated genus 0 {\em correlator}
is
$$
\blangle a_1, \dots, a_n \brangle^r =\int_{\oM_{0,n}} W^r_{0,n}(a_1, \dots, a_n)
\,\in \mathbb{Q}\,  .
$$
\end{definition}

\vspace{6pt}
\noindent Since $r$ can be deduced from $a_1,\ldots,a_n$, we will often drop
the superscript and denote the {correlator} by 
$\langle a_1, \dots, a_n \rangle$.

\subsection{The WDVV equation.}
Let $a,b,c,d,x_1,\ldots,x_k \in \{ 0, \dots, r-2 \}$ satisfy
$$a + b + c + d + \sum_{i=1}^k x_i = (k+1)r-2\, ,$$
 so $W^r_{0,4+k}(a,b,c,d,x_1,\ldots,x_k)$ is a class of 
 degree
$k$ on $\oM_{0,4+k}$.
Since the dimension of $\oM_{0,4+k}$ is $k+1$,
 we can cut $W^r_{0,k+4}(a,b,c,d,x_1,\ldots,x_k)$
once  with  boundary divisors. 
If we pull-back the WDVV relation from $\oM_{0,4}$ 
to $\oM_{4+k}$ via
the map forgetting the last $k$ markings, 
$$p_k:\oM_{4+k} \rightarrow \oM_{4}\, ,$$
and apply the CohFT splitting axiom for Witten's class,   
we obtain
\begin{equation} \label{Eq:WDVV}
\sum_{\substack{I\sqcup J = \{1,\ldots,k\} \\ * + \widehat{*} = r-2}}\blangle x_I,a,c,*\brangle \blangle x_J,b,d,\widehat{*} \brangle = \sum_{\substack{I\sqcup J = \{1,\ldots,k\} \\ * + \widehat{*} = r-2}} \blangle x_I,a,d,*\brangle \blangle x_J,b,c,\widehat{*}\brangle\, ,
\end{equation}
where $x_I$ and $x_J$ are the insertions 
$$(x_i : i \in I) \ \ \ \text{and} \ \ \ (x_j : j \in J)\, ,$$  
and $*$ and $\widehat{*}$ are non-negative integers with sum $r-2$.

Using the values \eqref{fred} of the 3-points correlators, 
we can rewrite the WDVV equation~\eqref{Eq:WDVV} as
\begin{multline*}
\langle a,c, b+d, x_1, \dots, x_k \rangle
+
\langle b,d, a+c, x_1, \dots, x_k \rangle\\
=
\langle a,d, b+c, x_1, \dots, x_k \rangle
+
\langle b,c, a+d, x_1, \dots, x_k \rangle
+ \mathsf{Q},
\end{multline*}
where, by convention, a correlator vanishes if any insertion exceeds $r-2$.
Here, $\mathsf{Q}$ is a sum of products 
of correlators with $4$ to $k+2$ insertions each. 
In particular, $\mathsf{Q}$ vanishes unless $k\geq 2$.

\begin{lemma} \label{455t}
The $3$-point evaluations
 \begin{equation*}
\blangle a_1,a_2,a_3\brangle = 
\left|
\begin{array}{cl}
1 & \mbox{ {\em if} } a_1+a_2+a_3 = r-2,\\
0 & \mbox{ {\em otherwise}}
\end{array}
\right.
\end{equation*}
together
with  the WDVV equation \eqref{Eq:WDVV} force the vanishing of all 
genus 0 correlators 
$$\blangle a_1, \dots, a_n \brangle$$ with $n \geq 4$ 
and at least one $0$ insertion.
\end{lemma}

\begin{proof} Apply the WDVV equation to the class
$W^r_{0,n+1}(a,0,b,0, c_1, \dots, c_{n-3})$ with nonzero $a$ and~$b$. We obtain
\begin{multline*}
\langle a+b,0, 0,c_1, \dots, c_{n-3} \rangle
+
\langle a, b, 0,c_1, \dots, c_{n-3} \rangle\\
=
\langle a, b, 0, c_1, \dots, c_{n-3} \rangle
+
\langle a, b, 0, c_1, \dots, c_{n-3} \rangle
+ \mathsf{Q}
\end{multline*}
or, equivalently,
\begin{equation} \label{f23f}
\langle a, b, 0,c_1, \dots, c_{n-3} \rangle = 
\langle a+b,0, 0,c_1, \dots, c_{n-3} \rangle - \mathsf{Q}.
\end{equation}

If $n=4$, the quadratic term $\mathsf{Q}$ vanishes. Hence, we can
use \eqref{f23f} to increase the number of $0$ insertions from 1
to 2 and then 2 to 3. The nonzero insertion of the correlator
$\blangle x, 0,0,0 \brangle$
must satisfy
$$ x = 2r-2> r-2\, .$$
Hence $\blangle x, 0,0,0 \brangle=0$, and the Lemma is proved for $n=4$.

Assume $n\geq 5$.
In every term of $\mathsf{Q}$ in \eqref{f23f},
 there is a correlator with fewer than $n$ 
insertions and with a 0 insertion. 
By induction, such correlators vanish, so $\mathsf{Q}$
again vanishes.
We may use \eqref{f23f} as before to increase the
number of 0 entries until 
we arrive at 
\begin{equation}\label{nvvn}
\blangle x, \underbrace{0,0,0}_{n-1} \brangle 
\end{equation}
%
Then $x=(n-2)r-2$ exceeds $r-2$, and the
correlator \eqref{nvvn} vanishes.  
\end{proof}

\begin{lemma} \label{nn33}
The correlators 
 \begin{equation*}
\blangle a_1,a_2,a_3\brangle = 
\left|
\begin{array}{cl}
1 & \mbox{ {\em if} } a_1+a_2+a_3 = r-2,\\
0 & \mbox{ {\em otherwise} }
\end{array}
\right.
\end{equation*}
$$
\hspace{-180pt}\blangle r-2,r-2,1,1\brangle  = \frac1{r} 
$$
%
and the WDVV equation \eqref{Eq:WDVV} uniquely determine the values of 
{\em all} the other correlators in genus 0.
\end{lemma}

\begin{proof} Consider a correlator with $n\geq 4$ insertions. 
If the correlator has at least one 0 insertion, 
the correlator vanishes by Lemma \ref{455t}. If not, all
insertions are {\em positive}.

Let $\langle a, b, c, d_1, \dots, d_{n-3} \rangle$ 
be a correlator with positive insertions ordered in a non-increasing sequence. 
Assume also $c \geq 2$. For $n \geq 5$, the condition
$c\geq 2$ is automatic, while for $n=4$, the condition excludes only the
 single correlator 
$$\langle r-2, r-2, 1,1 \rangle = \frac{1}{r}\, .$$

By applying the WDVV equation \eqref{Eq:WDVV}  to Witten's $r$-spin class
$$
W^r_{0,n+1}(a,1,c-1,b, d_1, \dots, d_{n-3})\, ,
$$
we obtain the relation
\begin{multline} \label{x33x} 
\langle a+c-1, b, 1, d_1, \dots, d_{n-3} \rangle
+
\langle a, b+1, c-1, d_1, \dots, d_{n-3} \rangle \\
=
\langle a+b, c-1, 1, d_1, \dots, d_{n-3} \rangle
+
\langle a, b, c, d_1, \dots, d_{n-3} \rangle
+\mathsf{Q}\, .
\end{multline}
If we order the insertions of the correlators in non-increasing order, 
then our original correlator $\langle a, b, c, d_1, \dots, d_{n-3} \rangle$ is
lesser than the three others principal correlators of \eqref{x33x}
in  lexicographic order. Thus every correlator except $\langle r-2, r-2, 1,1 \rangle$ can be expressed via correlators greater in 
 lexicographic order and correlators with fewer entries. 
\end{proof}

\subsection{Proof of Theorem~\ref{Thm:sl2}.} 
A straightforward calculation of the $\sl_2$ side shows that the formula of
Theorem \ref{Thm:sl2} is true for $n \le 4$.
By Lemma \ref{nn33},
  the WDVV equation \eqref{Eq:WDVV}
uniquely determines the correlators for $n \ge 5$. 
To finish the proof, we need only to show that the
$\sl_2$ side of the formula of Theorem \ref{Thm:sl2}
 satisfies the WDVV equation \eqref{Eq:WDVV}.

The WDVV equation \eqref{Eq:WDVV} starts with a choice of insertions
$$a,b,c,d,x_1,\ldots,x_k\, .$$
For convenience, we replace the above insertions
by their  $r-2$ complements:
$$a \mapsto r-2-a\, ,\ \ \ b \mapsto r-2-b\, , \ \ \ c\mapsto r-2-c\, ,
\ \ \ d \mapsto r-2-d\, , \ \ \ x_i \mapsto r-2-x_i\, .$$
Then, we have
\[
a + b + c + d + \sum_{i = 1}^kx_i = 3r-6-2k\, .
\]
We must prove the identity 
\begin{equation*}
\begin{split}
\sum_{I\sqcup J = \{1,\ldots,k\}} |I|!&|J|!\Blangle x_I,a,c,2r-4-a-c-2|I|-\sum_{i\in I}x_i\Brangle^{\sl_2}\\
&\hspace{8pt} \cdot \Blangle x_J,b,d,2r-4-b-d-2|J|-\sum_{j\in J}x_j\Brangle^{\sl_2}\\
= \sum_{I\sqcup J = \{1,\ldots,k\}} |I|!&|J|!\Blangle x_I,a,d,2r-4-a-d-2|I|-\sum_{i\in I}x_i\Brangle^{\sl_2}\\
&\hspace{8pt} \cdot \Blangle x_J,b,c,2r-4-b-c-2|J|-\sum_{j\in J}x_j\Brangle^{\sl_2},
\end{split}
\end{equation*}
where we define
\[
\Blangle a_1,\ldots,a_k\Brangle^{\sl_2} = \dim \Bigl[ \rho_{a_1} \otimes \cdots \otimes \rho_{a_k} \Bigr]^{\sl_2}\, .
\]

The following formula for the dimensions of the $\sl_2$-invariants follows easily from the multiplication rule for the $\rho_a$:
\[
\Blangle a_1,\ldots,a_n,2s-a_1-\cdots-a_n \Brangle^{\sl_2} = \left[(1-t^{-1})\prod_{i=1}^n\frac{1-t^{a_i+1}}{1-t}\right]_{t^s}.
\]
Applying the above formula, we can rewrite the desired identity as
{\footnotesize{
\begin{equation*}
\begin{split}
&\sum_{I\sqcup J = \{1,\ldots,k\}} \!\!\!\! |I|!|J|!\left[
\frac{(1-t^{a+1})(1-t^{c+1})(1-u^{b+1})(1-u^{d+1})}{(1-t)(1-u)}\prod_{i\in I}\frac{t-t^{x_i+2}}{1-t}
\prod_{j\in J}\frac{u-u^{x_j+2}}{1-u}
\right]_{t^{r-1}u^{r-1}}
\\
&= \!\!\!\! \sum_{I\sqcup J = \{1,\ldots,k\}}  \!\!\!\! |I|!|J|!\left[
\frac{(1-t^{a+1})(1-t^{d+1})(1-u^{b+1})(1-u^{c+1})}{(1-t)(1-u)}\prod_{i\in I}\frac{t-t^{x_i+2}}{1-t}
\prod_{j\in J}\frac{u-u^{x_j+2}}{1-u}
\right]_{t^{r-1}u^{r-1}}.
\end{split}
\end{equation*}}}

Subtracting one side from the other and moving factors outside the sum, we 
must, equivalently, show that the coefficient of $t^{r-1}u^{r-1}$ in
\begin{equation*}
\begin{split}
(1-t^{a+1})&(1-u^{b+1})\frac{(1-t^{c+1})(1-u^{d+1})-(1-t^{d+1})(1-u^{c+1})}{(1-t)(1-u)} \\
&\cdot \sum_{I\sqcup J = \{1,\ldots,k\}} |I|!|J|!\prod_{i\in I}\frac{t-t^{x_i+2}}{1-t}
\prod_{j\in J}\frac{u-u^{x_j+2}}{1-u}
\end{split}
\end{equation*}
vanishes. Interchanging $t$ and $u$ and adding, we can replace the
latter polynomial by a symmetric one:
{\footnotesize{
\begin{equation*}
\left((1-t^{a+1})(1-u^{b+1})-(1-t^{b+1})(1-u^{a+1})\right) \!
 \frac{(1-t^{c+1})(1-u^{d+1})-(1-t^{d+1})(1-u^{c+1})}{(1-t)(1-u)}
\end{equation*}}}
$$\cdot \sum_{I\sqcup J = \{1,\ldots,k\}} |I|!|J|!\prod_{i\in I}
\frac{t-t^{x_i+2}}{1-t}
\prod_{j\in J}\frac{u-u^{x_j+2}}{1-u}\, .
$$
Next, we apply the change of variables
 $$t\mapsto \frac{t}{v} \ \ \ {\text{and}} \ \ \ u \mapsto \frac{u}{v}$$ and 
multiply the outcome by $v^m$ where 
$$m = a+b+c+d+3+\sum_{i=1}^k(x_i+2) = 3r-3\, .$$
Then, we replace $a,b,c,d$ by $a-1,b-1,c-1,d-1$ and $x_i$ by $x_i-2$. 
After these transformations,
the required identity is that the coefficient of 
$t^{r-1}u^{r-1}v^{r-1}$ vanishes in the polynomial
{\footnotesize{
$$
\frac{(t^au^b-u^at^b-t^av^b+v^at^b+u^av^b-v^au^b)(t^cu^d-u^dt^c-t^cv^d+v^ct^d+u^cv^d-v^cu^d)}{t^2u-u^2t-t^2v+v^2t+u^2v-v^2u}
$$}}
$$\cdot (t-u)v\sum_{I\sqcup J = \{1,\ldots,k\}} |I|!|J|!\prod_{i\in I}
\frac{t^{x_i}v-v^{x_i}t}{t-v}
\prod_{j\in J}\frac{u^{x_j}v-v^{x_j}u}{u-v}\, .
$$

The initial factor is invariant under cyclically permuting $t,u,v$. 
So we can instead prove the vanishing of the
 coefficient of $t^{r-1}u^{r-1}v^{r-1}$ in the polynomial
{\footnotesize{
$$
\frac{(t^au^b-u^at^b-t^av^b+v^at^b+u^av^b-v^au^b)(t^cu^d-u^dt^c-t^cv^d+v^ct^d+u^cv^d-v^cu^d)}{t^2u-u^2t-t^2v+v^2t+u^2v-v^2u}
$$}}
$$
\cdot \sum_{\text{cyc}}(t-u)v\sum_{I\sqcup J = \{1,\ldots,k\}} |I|!|J|!\prod_{i\in I}\frac{t^{x_i}v-v^{x_i}t}{t-v}
\prod_{j\in J}\frac{u^{x_j}v-v^{x_j}u}{u-v}\,  ,
$$
where the sum in the second term is
over the three cyclic permutations of $t,u,v$.
In fact, we claim a stronger vanishing:
\begin{equation}\label{cyclic identity}
\sum_{\text{cyc}}(t-u)v\sum_{I\sqcup J = \{1,\ldots,k\}} |I|!|J|!
\prod_{i\in I}\frac{t^{x_i}v-v^{x_i}t}{t-v}
\prod_{j\in J}\frac{u^{x_j}v-v^{x_j}u}{u-v} = 0
\end{equation}
for any $x_i\in\Z$.

The last step in the proof of Theorem \ref{Thm:sl2} is to
show the vanishing \eqref{cyclic identity}.
Let $A\sqcup B\sqcup C = \{1,\ldots,k\}$ be a partition of the $\{x_i\}$ into three sets of 
sizes 
$$a = |A|\,,\ \ \ b = |B|\, ,\ \ \  c = |C|\, .$$
 We can compute the coefficient of $\prod_{i\in A}t^{x_i}\prod_{i\in B}u^{x_i}\prod_{i\in C}v^{x_i}$ on
the left side \eqref{cyclic identity}:
\begin{equation*}
\sum_{\text{cyc}}(t-u)v\sum_{i+j = c}\binom{c}{i}(i+a)!(j+b)!\left(\frac{v}{t-v}\right)^a\left(\frac{v}{u-v}\right)^b\left(\frac{-t}{t-v}\right)^i\left(\frac{-u}{u-v}\right)^j,
\end{equation*}
where now $a,b,c$ are also cyclically permuted in correspondence with $t,u,v$.
Multiplying by $(-1)^{a+b+c}(t-u)^{a+b}(u-v)^{b+c}(v-t)^{c+a}$ and dividing by $a!b!c!$ then yields
\begin{equation} \label{kqq2}
\sum_{\text{cyc}}\sum_{i+j = c}(-1)^{b+i}\binom{a+i}{i}\binom{b+j}{j}t^iu^jv^{a+b+1}(t-u)^{a+b+1}(u-v)^i(v-t)^j\, .
\end{equation}

Now set $z_1 = t(u-v)$, $z_2 = u(v-t)$, $z_3 = v(t-u)$, multiply by $y_1^ay_2^by_3^c$, and sum 
 \eqref{kqq2} over all non-negative integers $a,b,c$. The result is
\begin{equation*}
\sum_{\text{cyc}}\frac{z_3}{(1+z_1y_3-z_3y_1)(1+z_3y_2-z_2y_3)}\, ,
\end{equation*}
which expands to
\begin{equation} \label{eeww}
\frac{z_1+z_2+z_3}{(1+z_1y_3-z_3y_1)(1+z_2y_1-z_1y_2)(1+z_3y_2-z_2y_3)}\, .
\end{equation}
Since $z_1+z_2+z_3 = tu-tv+uv-ut+vt-vu = 0$, so \eqref{eeww} vanishes, and the
identity \eqref{cyclic identity} holds. \qed

\bigskip


\begin{proposition} \label{Cor:vanishing}
A genus 0, $n$-point correlator vanishes if there is an insertion  less than $n-3$.
\end{proposition}

\begin{proof} Consider the correlator 
\begin{equation}\label{ppaapp}
\blangle a_1, \dots, a_n \brangle\,, \ \   \ \ \ \ \sum_{i=1}^n a_i = r(n-2)-2\, .
\end{equation}
 Let $b_i = r-2-a_i$, so  
$\sum_{i=1}^n b_i = 2(r+1-n)$. If one of the $b_i$ is greater than the sum of all others, 
then the representation
$\otimes_{i=1}^n \rho_{b_i}$
has no invariant vectors and the correlator \eqref{ppaapp} vanishes by Theorem~\ref{Thm:sl2}. 
Thus the greatest possible value of $b_i$ in a nonzero correlator is equal to $r+1-n$. In other words, the smallest possible value of $a_i$ is $n-3$. 
\end{proof}

\section{The semisimple point $\tau = (0, \dots, 0, r \phi)$}
\label{Sec:lastshift}

\subsection{Special shift}
The last basis vector $e_{r-2}\in V_r$ plays a special role in Witten's $
r$-spin theory: $e_{r-2}$  corresponds,
 via the formula of Theorem~\ref{Thm:sl2},
to the trivial representation
of $\sl_2$. The shift along the last coordinate yields simpler expressions for the quantum product and the $R$-matrix which are calculated
in Sections \ref{d11d} and \ref{Ssec:Rmatrix1} respectively.
The associated topological field theory is
calculated in Sections \ref{Ssec:TopFT1} and \ref{Ssec:idempotents}.

\subsection{The quantum product (proof of Theorem~\ref{Thm:Verlinde})}
\label{d11d}

We study the Frobenius algebra at the point
$$\tau = (0, \dots, 0, r \phi)\in V_r$$
and find an isomorphism with the Verlinde fusion algebra for $\sl_2$ of level $r$. 

Let $a+b+c = r-2+2k$. By Theorem~\ref{Thm:sl2}, we have
\begin{equation}\label{nnee33}
W^r_{0,3+k}(a,b,c,\underbrace{r-2, \dots, r-2}_k) = 
\left|
\begin{array}{ll}
\frac{k!}{r^k} & \mbox{if} \quad  \min(a,b,c) \geq k,\\
0 & \mbox{otherwise}.
\end{array}
\right.
\end{equation}
The {\em quantum product} at $\tau$ is given  by
$$
\d_a \bullet_\tau \d_b = 
\sum_{k=\max(0,a+b-r+2)}^{\min(a,b)}
\phi^k \d_{a+b-2k}\, .
$$
The $k!$ in the evaluation \eqref{nnee33} is cancelled by the $k!$ in the denominator of the Gromov-Witten potential,
 and  the $r^k$ in the evaluation is cancelled by $r$ factor in the last coordinate of $\tau$. 

To simplify the computations, we introduce the new frame{\footnote{The metric
is $\eta( \hd_a,\hd_b) = \delta_{a+b,r+2}$.}}
$$
\hd_a = \phi^{-(2a-r+2)/4} \d_a\, .
$$
Then, the quantum product assumes a slightly simpler form:
\begin{equation} \label{Eq:QuantumProd1}
\hd_a \bullet_\tau \hd_b = 
\phi^{(r-2)/4} \sum_{k=\max(0,a+b-r+2)}^{\min(a,b)}
\hd_{a+b-2k}.
\end{equation}
We have proven the following result.

\begin{proposition}\label{Prop:QP}
The coefficient of $\hd_c$ in the quantum product 
$\hd_a \bullet_\tau \hd_b$ at the point $\tau$ is equal to either to $\phi^{(r-2)/4}$ or to~$0$. 
The coefficient is equal to $\phi^{(r-2)/4}$ if and only if
 $a+b+c$ is even and the point $(a,b,c)$ lies inside the tetrahedron 
\begin{align*}
& a+b \geq c, \\
& a+c \geq b, \\
& b+c \geq a, \\
& a+b+c \leq 2r-4.
\end{align*}
\end{proposition}

The tetrahedron is represented in Figure~\ref{Fig:tetrahedron}.
To help the reader visualize the constraint in Propositions \ref{Prop:QP},
 we have represented on the right the values of $a$ and $b$ for which
 the quantum product $\hd_a \bullet \hd_b$ contains the term $\hd_c$ for a fixed $c$ (a horizontal section of the tetrahedron with the parity condition taken into account).

\begin{figure}[h]
\begin{center}
\ 
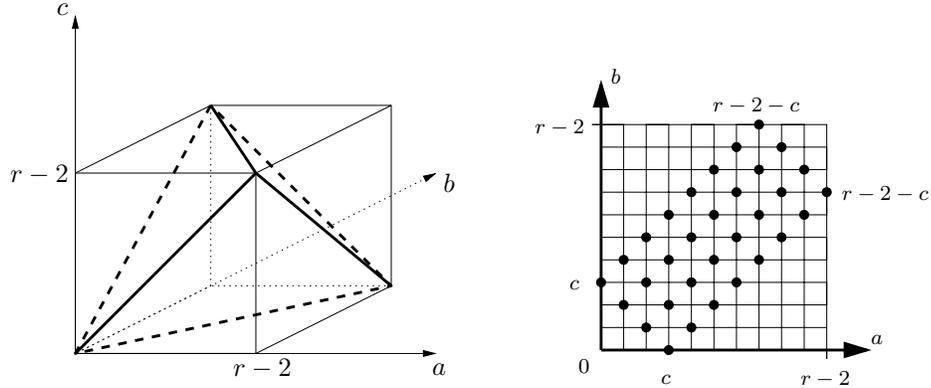
\end{center}
\caption{Quantum product: nonzero coefficients} \label{Fig:tetrahedron}
\end{figure}

The {\em Verlinde algebra} of level~$r$ for $\sl_2$ is spanned by the weights of $\sl_2$ from $0$ to $r-2$. The coefficient of $c$ in the product $a \bullet b$ is equal to the dimension of the $\sl_2$-invariant subspace of the 
representation $\rho_a \otimes \rho_b \otimes \rho_c$ provided the inequality 
$$a+b+c \leq 2r-4$$ is satisfied. 
The invariant subspace is easily seen to always have dimension 0 or 1. It is of dimension 1 if
and only if  $a+b+c$ is even and $(a,b,c)$ satisfy the triangle inequalities. 
The structure constants of the Verlinde algebra therefore coincide with those of the Frobenius algebra at 
$\tau = (0, \dots, 0, r)$. We have completed the proof of Theorem \ref{Thm:Verlinde}. \qed

\vspace{10pt}

Although the algebra structure of the Verlinde and Frobenius algebras are the same, 
the quadratic forms are different. The quadratic form for the Verlinde algebra is given by 
$$\eta^{\text{Ver}}_{a,b} = \delta_{a,b}\, ,$$
 while for the Frobenius manifold $V_r$ we have
 $$\eta^{\text{Frob}}_{a,b} = \delta_{a+b,r-2}\, .$$
The difference is related to the large automorphism group of the tetrahedron. The symmetry under permutations of $a, b, c$ of the structure constants of the algebra is expected. However, the tetrahedron has an extra symmetry obtained by replacing every $a$ by $r-2-a$ (central symmetry) and then changing the metric from $\eta^{\text{Ver}}$ to $\eta^{\text{Frob}}$ (a vertical flip).

\subsection{The $R$-matrix} \label{Ssec:Rmatrix1}
We compute here the $R$-matrix of the Frobenius manifold $V_r$ 
at the point $$\tau = (0, \dots, 0, r \phi)\, .$$
We do not know closed form expressions for the 
 coefficients of the $R$-matrix at {\em any} other point of $V_r$. 

The operator of quantum multiplication by the Euler field at $\tau$,
$$E= 2 \phi^{(r+2)/4}\, \hd_{r-2}\, ,$$ 
is  given 
in the frame $\{\hd_a\}$ by the matrix
$$
\xi = 
2 \phi^{r/2} 
\begin{pmatrix}
\nice 0 & \nice \cdots & \nice \cdots & \nice 0  & \nice 1\\
 \qquad  & \qquad & \qquad & \qquad & \qquad \\
\nice 0 & & \nice 0 & \nice 1 & 0 \\
\\
\nice \vdots & \nice \iddots & \nice \iddots  & \nice \iddots & \nice \vdots \\
\\
\nice 0 & \nice 1 & \nice 0  &  & \nice 0 \\
\\
\nice 1 & \nice 0 & \nice \cdots  & \nice \cdots & \nice 0 
\end{pmatrix}\, .
$$

\vspace{5pt}
\noindent The nonvanishing coefficients{\footnote{In our convention
for matrix coefficients, the superscript is the row
index and the subscript is the column index.}}  are
$$
\xi^a_{r-2-a} = 2 \phi^{r/2}\, .
$$
In the same frame, the shifted degree operator is
$$
\mu = \frac1{2r}
\begin{pmatrix}
\nice -(r-2) & \nice 0 & \nice \cdots & \nice \cdots & \nice 0  \\
 \qquad & \qquad  & \qquad  & \qquad & \qquad \\
\nice 0 & -(r-4) & \nice 0 &  & 0 \\
\\
\nice \vdots & \nice \ddots & \nice \ddots  & \nice \ddots & \nice \vdots \\
\\
\nice 0 &  & \nice 0  & r-4 & \nice 0 \\
\\
\nice 0  & \nice \cdots  & \nice \cdots & \nice 0 & \nice r-2 
\end{pmatrix}\, .
$$

\vspace{5pt}
\noindent The nonvanishing coefficients are
$$
\mu^a_a = \frac{2a-r+2}{2r}\, .
$$

Recall the hypergeometric series defined in Section \ref{trtr} for 
every integer $a \in \{0, \dots, r-2 \}$, 
$$
\B_{r,a}(T) = \sum_{m=0}^\infty
\left[
\prod_{i=1}^m 
\frac{\big((2i-1)r-2(a+1)\big)\big((2i-1)r+2(a+1)\big)}i 
\right] \!\! 
\left( - \frac{T}{16r^2} \right)^{\! m} \!\!\!.
$$
We denote by $\Be_{r,a}$ and $\Bo_{r,a}$ the even and odd parts of the 
series~$\B_{r,a}$.

\begin{proposition}
The unique solution $R(z)=\sum_{m=0}^\infty R_m z^m \in \End(V_r)[[z]]$ of the equations
$$
[R_{m+1},\xi] = (m+\mu) R_m
$$
with the initial condition $R_0=1$ has coefficients
$$
R^a_a = \Be_{r,r-2-a}(\phi^{-r/2}z)\,, \ \ \ \  a \in \{0, \dots, r-2 \}
$$
on the main diagonal,
$$
R^{r-2-a}_a =  -\Bo_{r,a}(\phi^{-r/2}z)\,, \ \ \ \  
a \in \{0, \dots, r-2 \}
$$
on the antidiagonal (if $r$ is even, the coefficient at the intersection of both diagonals is 1), and $0$ everywhere else. 

The inverse matrix $R^{-1}(z)$ has coefficients
$$
(R^{-1})^a_a = \Be_{r,a}(\phi^{-r/2}z)\,, \ \ \ \ 
 a \in \{0, \dots, r-2 \}
$$
on the main diagonal,
$$
(R^{-1})^{r-2-a}_a = \Bo_{r,a}(\phi^{-r/2}z)\,, \ \ \ \ 
 a \in \{0, \dots, r-2 \}
$$
on the antidiagonal (if $r$ is even, the coefficient at the intersection of both diagonals is 1), and $0$ everywhere else. 
\end{proposition}


\begin{proof}
The {\em uniqueness} of the solution follows from the semisimplicity of 
the Frobenius manifold $V_r$ at $\tau$ (proven
in Section \ref{Ssec:TopFT1}). We verify here that the $R$-matrix described in 
the Proposition is indeed a solution of the recursion
\begin{equation}\label{fpp2}
[R_{m+1},\xi] = (m+\mu) R_m
\end{equation}
with initial condition{\footnote{Here, $R_0=1$ denotes the
identity in $\End(V_r)$.}} $R_0=1$.

The coefficients of the commutator on the left side of
\eqref{fpp2} are given by
\begin{equation}\label{vvww}
[R_{m+1}, \xi]^a_b = 2\phi^{r/2} \Big( 
(R_{m+1})^a_{r-2-b} - (R_{m+1})^{r-2-a}_b \Big)\, .
\end{equation}
The right side of \eqref{vvww} vanishes unless $a=b$ or $a+b=r-2$. 
The same is true of $R_m$ and therefore of $(m+\mu)R_m$, since $\mu$ is a diagonal matrix. Thus, we have two cases to consider.

\vspace{10pt}
\noindent {\bf Case $a=b$.} We have
\begin{eqnarray*}
[R_{m+1}, \xi]_a^a& = & 
2\phi^{r/2} \big( (R_{m+1})^a_{r-2-a} - (R_{m+1})^{r-2-a}_a\big) \\
&
=&
2\phi^{r/2} [z^{m+1}] \big( -\Bo_{r,r-2-a}(\phi^{-r/2}z) + \Bo_{r,a}(\phi^{-r/2} z) \big) \, . 
\end{eqnarray*}
Using the definition of the $\Bo$ series, the last expression is
\begin{multline*}
  \delta^{\rm even}_m\, 2\phi^{r/2}  
\prod_{i=1}^{m+1}
\frac{\Big((2i-1)r-2(a+1)\Big)\Big((2i-1)r+2(a+1)\Big)}i 
\\
\cdot 
\left(\frac{(3+2m)r-2(a+1)}{(2m+1)r+2(a+1)}+1 \right)
\left( - \frac1{16r^2\,\phi^{r/2}} \right)^{m+1}\, ,
\end{multline*}
or, equivalently,
\begin{multline*}
\delta^{\rm even}_m \, 2\phi^{r/2}  \prod_{i=1}^{m+1}
\frac{\Big((2i-1)r-2(a+1)\Big)\Big((2i-1)r+2(a+1)\Big)}i  \\
\cdot 
\left(\frac{4 (m+1)r}
{(2m+1)r+2(a+1)}\right) \left( - \frac1{16r^2\,\phi^{r/2}} \right)^{m+1}\, .
\end{multline*}
 After further simplification, we conclude
$$
[R_{m+1}, \xi]_a^a=
 - \delta^{\rm even}_m \,  
\left(\frac{(2m+1)r-2(a+1)}{2r}\right)
\cdot [z^m] \B_{r,a}\, .
$$
On the other hand, we have
\begin{eqnarray*}
[(m+\mu)R_m]^a_a  = \hspace{-6em}\\
&=& \left( m+\frac{2a-r+2}{2r} \right) (R_m)^a_a \\
&=& \delta_m^{\rm even}\, \left( \frac{(2m-1)r+2a+2}{2r}\right) 
\cdot [z^m] \B_{r,r-2-a} \\
&=&\delta_m^{\rm even} \, \left( \frac{(2m-1)r+2a+2}{2r}\right)  
\left(\frac{2(a+1)-(1+2m)r}{2(a+1)-(1-2m)r}\right)
\cdot [z^m] \B_{r,a}  \\
&=& - \delta^{\rm even}_m\, 
\left(\frac{(2m+1)r-2(a+1)}{2r}\right)
\cdot [z^m] \B_{r,a}\, .
\end{eqnarray*}
Therefore, equation \eqref{fpp2} is satisfied.

\vspace{10pt}
\noindent{\bf Case $a+b=r-2$.} We have
\begin{eqnarray*}
[R_{m+1}, \xi]^{r-2-a}_a & =& 
2\phi^{r/2} \big((R_{m+1})^{r-2-a}_{r-2-a} - (R_{m+1})^a_a\big) \\
&=& 
2\phi^{r/2} [z^{m+1}] \big( \Be_{r,r-2-a}(\phi^{-r/2}z) - \Be_{r,a}(\phi^{-r/2} z) \big)\, .
\end{eqnarray*}
Using the definition of the $\Bo$ series, the last expression is
\begin{multline*}
\delta^{\rm odd}_m\, 2\phi^{r/2}  
 \prod_{i=1}^{m+1}
\frac{\Big((2i-1)r-2(a+1)\Big)\Big((2i-1)r+2(a+1)\Big)}i \\
\cdot \left(\frac{2(a+1)-(3+2m)r}{2(a+1)+(2m+1)r}-1 \right)
\left( - \frac1{16r^2\,\phi^{r/2}} \right)^{m+1}\, ,
\end{multline*}
or, equivalently,
\begin{multline*}
-\delta^{\rm odd}_m \, 2\phi^{r/2}  \prod_{i=1}^{m+1}
\frac{\Big((2i-1)r-2(a+1)\Big)\Big((2i-1)r+2(a+1)\Big)}i  \\
\cdot 
\left(\frac{4 (m+1)r}
{(2m+1)r+2(a+1)}\right) \left( - \frac1{16r^2\,\phi^{r/2}} \right)^{m+1}\, .
\end{multline*}
 After further simplification, we conclude
$$
[R_{m+1}, \xi]_a^a=
  \delta^{\rm odd}_m \,  
\left(\frac{(2m+1)r-2(a+1)}{2r}\right)
\cdot [z^m] \B_{r,a}\, .
$$
On the other hand, we have
\begin{eqnarray*}
[(m+\mu)R_m]^{r-2-a}_a &=& \left( m-\frac{2a-r+2}{2r} \right) (R_m)^{r-2-a}_a
\\
&=& \delta_m^{\rm odd} \frac{(2m+1)r-2(a+1)}{2r} \cdot [z^m] \B_{r,a}\, .
\end{eqnarray*}
Therefore, equation \eqref{fpp2} is satisfied.

\vspace{10pt}
The expression for $R^{-1}$ is obtained from the symplectic 
condition{\footnote{$R^*$ denotes the adjoint with respect
to the metric $\eta$ on $V_r$.}} 
$$R^{-1}(z) = R^*(-z)\, $$
for the $R$-matrix.
The symplectic condition implies the identities
$$
\B_{r,a}(T) \B_{r,r-2-a}(-T) + \B_{r,a}(-T) \B_{r,r-2-a}(T) =2,
$$
or, equivalently,
$$
\Be_{r,a}(T) \Be_{r,r-2-a}(T) - \Bo_{r,a}(T) \Bo_{r,r-2-a}(T) = 1.
$$
Of course, these identities can also be proved directly. 
\end{proof}

\subsection{The topological field theory} \label{Ssec:TopFT1}
We study next the topological field theory arising from the Frobenius algebra of $V_r$ at the point
 $$\tau = (0, \dots, 0, r \phi)\, .$$

\begin{proposition} \label{Prop:idempotents}
The basis of normalized idempotents of the quantum product~\eqref{Eq:QuantumProd1} is given by
$$
v_k = \sqrt{\frac2r} \; \sum_{a=0}^{r-2} \sin \left(
\frac{(a+1)k\pi}r \right) \; \hd_a, \qquad k\in\{ 1, \dots, r-1\}\, .
$$
More precisely, we have
$$
\eta(v_k, v_l) = (-1)^{k-1} \delta_{k,l}\, ,
\qquad \ \
v_k \bullet_\tau v_l = \phi^{(r-2)/4} 
\frac{\sqrt{r/2}}{\sin(\frac{k \pi}{r})} v_k \, \delta_{k,l}\, .
$$
\end{proposition}

\begin{proof} The idempotents of the Verlinde algebras are known for all semisimple Lie algebras and all levels~\cite{Verlinde}. The idempotents are automatically orthogonal to each other, so we 
only need to compute their scalar squares. Since the extraction of the 
elementary formulas for $\sl_2$ from the general case is not so simple, 
we check the statements of the Proposition independently by a series of 
elementary computations with 
trigonometric functions in Section~\ref{Ssec:idempotents}. 
\end{proof}

\begin{proposition} \label{Prop:QuantumProd1}
For $a_1, \dots, a_n \in \{0, \dots, r-2 \}$,
we have
\begin{multline*}
\omega^{r,\tau}_{g,n}(\hd_{a_1} \otimes \cdots \otimes \hd_{a_n}) =\\
\left(\frac{r}2\right)^{g-1}
\phi^{(r-2)(2g-2+n)/4} \; 
\sum_{k=1}^{r-1}
\frac
{(-1)^{(k-1)(g-1)}\prod\limits_{i=1}^n \sin \left(
\frac{(a_i+1)k\pi}r \right)}
{\left(\sin(\frac{k\pi}{r})\right)^{2g-2+n}}\, .
\end{multline*}
\end{proposition}

\begin{proof}
The topological field theory $\omega^{r,\tau}_{g,n}$ can be computed
by restricting the $\tau$-shifted $r$-spin theory $\W^{r,\tau}_{g,n}$ 
to $$[C]\in \oM_{g,n}\, ,$$
where $C$ is a completely degenerate curve with 
$2g-2+n$ rational components  
and $3g-3+n$ nodes.

By Proposition \ref{Prop:idempotents}, 
$\omega^{r,\tau}_{0,3}(v_i \otimes v_j \otimes v_k)$ vanishes unless the three indices coincide. 
For equal indices, we have 
\begin{equation}\label{fppaa}
\omega^{r,\tau}_{0,3}(v_k \otimes v_k \otimes v_k) = (-1)^{k-1}
\phi^{(r-2)/4} \frac{\sqrt{r/2}}{\sin(\frac{k\pi}{r})}\, .
\end{equation}
After applying the splitting axioms of the CohFT to
$$\omega^{r,\tau}_{g,n}(v_k \otimes \cdots \otimes v_k)\Big|_{[C]}\, ,$$
each node contributes the sign $(-1)^{k-1}$.
Including the sign in \eqref{fppaa}  associated
to each rational component, the total sign is
$$
(-1)^{(k-1)(2g-2+n)} \cdot (-1)^{(k-1)(3g-3+n)} = (-1)^{(k-1)(g-1)}\, .
$$
Hence, for $k\in \{1,\ldots, r-1\}$,
\begin{equation}\label{qwqw}
\omega^{r,\tau}_{g,n}(v_k \otimes \cdots \otimes v_k) =
\left(\sqrt{\frac{r}2}\right)^{2g-2+n}
\phi^{(r-2)(2g-2+n)/4} \; 
\frac
{(-1)^{(k-1)(g-1)}}
{\left(\sin(\frac{k\pi}{r})\right)^{2g-2+n}}\, .
\end{equation}

The change of basis matrix from $\hd_a$ to $v_k$ is self-inverse.{\footnote{The
self-inverse property follows easily from the trigonometric
identities used to calculate $\eta(v_k,v_l)$ in Section \ref{Prop:idempotents}.}} Therefore,
 we have
$$
\hd_a = \sqrt{\frac2r} \; \sum_{k=1}^{r-1} \sin \left(
\frac{(a+1)k\pi}r \right) \; v_k\, .
$$
The Proposition then follows from \eqref{qwqw}.
\end{proof}

\subsection{Proof of Proposition \ref{Prop:idempotents}.}
\label{Ssec:idempotents}

We provide here direct proofs of the claims of Proposition
 \ref{Prop:idempotents}. The methods are via simple manipulation
of trigonometric functions, but we were not able to find an
elementary reference.

\paragraph{The scalar square $\eta(v_k,v_k)$.} 
We have
\begin{eqnarray*}
\eta(v_k, v_k) &=& 
\frac{2}{r} \sum_{a=0}^{r-2} 
\sin\left( \frac{(a+1)k\pi}r \right)
\sin\left( \frac{(r-1-a)k\pi}r \right)\\
&=&
(-1)^{k-1}\, \frac{2}{r} \sum_{a=0}^{r-2} 
\sin^2\left( \frac{(a+1)k\pi}r \right) \\
&=& (-1)^{k-1}\, 
\frac{1}{r} \sum_{a=0}^{r-2} 
\left( 1 - \cos\left( \frac{2(a+1)k\pi}r \right) \right) \\
&=&{(-1)^{k-1}}\, \frac{1}{r} \Big( (r-1) - (-1)\Big) = (-1)^{k-1}\, .
\end{eqnarray*}

\paragraph{The scalar product $\eta(v_k,v_l)$ for $k\neq l$.}
We have
\begin{eqnarray*}
\eta(v_k, v_l) & \! = \!&
\frac{2}r \sum_{a=0}^{r-2} 
\sin\left( \frac{(a+1)k\pi}r \right) \sin\left( \frac{(r-1-a)l\pi}r \right)
\\ & \!= \!&
(-1)^{l-1}\, \frac{2}r \sum_{a=0}^{r-2} 
\sin\left( \frac{(a+1)k\pi}r \right) \sin\left( \frac{(a+1)l\pi}r \right)
\\
&\! =\! &
(-1)^{l-1}\,  \frac{1}r \sum_{a=0}^{r-2} 
\left[ \cos \left( \! \frac{(a+1)(k-l)\pi}r \!\right) - \cos\left(\! \frac{(a+1)(k+l)\pi}r \!\right) \right] \\
&\! = \!&
{(-1)^{l-1} }\, \frac{1}{r} \Big( (-1) - (-1)\Big) = 0\, .
\end{eqnarray*}

\paragraph{The quantum square $v_k \bullet_\tau v_k$.} 
The coefficient of $\hd_c$ in $v_k \bullet_\tau v_k$ is equal to
$$
\frac{2 \phi^{(r-2)/4}}r \sum_{a,b} \sin\left( \frac{(a+1)k\pi}r \right) \sin\left( \frac{(b+1)k\pi}r \right),
$$
where the sum runs over the set shown in Figure~\ref{Fig:tetrahedron}. 
The sum can be conveniently reparameterized by the
change of variables
 $$p = a+b+2\, , \ \ \ q = a-b\, .$$
 We obtain the following expression for the coefficient of $\hd_c$ in~$v_k \bullet_\tau v_k$:
\begin{multline}\label{kqqk}
\frac{2\phi^{(r-2)/4}}r \, \sum_{p=c+2}^{2r-2-c} \sum_{q=-c}^c 
\sin\left( \frac{(p+q)k\pi}{2r} \right) \sin\left( \frac{(p-q)k\pi}{2r} \right)
\\
= \frac{\phi^{(r-2)/4}}r \, \sum_{p=c+2}^{2r-2-c} \sum_{q=-c}^c \left(
\cos\left( \frac{qk\pi}r \right)- \cos \left( \frac{pk\pi}r \right)
\right),
\end{multline}
where the sums run over only over integers $p$ and $q$
which have the same parity as~$c$. The second
line of \eqref{kqqk} can be rewritten as
\begin{equation*}
(r-c-1) \frac{\phi^{(r-2)/4}}r  \sum_{q=-c}^{c} 
\Re \exp\left(\frac{iqk\pi}r \right)
-
 (c+1) \frac{\phi^{(r-2)/4}}r \sum_{p=c+2}^{2r-2-c} 
\Re \exp\left(\frac{ipk\pi}r \right),
\end{equation*}
where $\Re$ denotes the real part. The first term equals
$$
(r-c-1)\frac{\phi^{(r-2)/4}}{r}\ \Re \frac
{\exp\left(\frac{i(c+1)k\pi}{r}\right) - \exp\left(-\frac{i(c+1)k\pi}{r}\right)}
{\exp(\frac{ik\pi}{r}) - \exp(-\frac{ik\pi}{r})}\, ,
$$
and the second term equals
$$
-(c+1)  \frac{\phi^{(r-2)/4}}{r}\ \Re \frac
{\exp\left(\frac{i(2r-1-c)k\pi}{r}\right) - \exp\left(\frac{i(c+1)k\pi}{r}\right)}
{\exp(\frac{ik\pi}{r}) - \exp(-\frac{ik\pi}{r})}.
$$
Since $\exp\left(\frac{i(2r-1-c)k\pi}{r}\right)$ and 
$\exp\left(-\frac{i(c+1)k\pi}{r}\right)$ are equal, the
first and second terms  
 combine as
$$
\phi^{(r-2)/4}\ \Re  \frac
{\exp\left(\frac{i(c+1)k\pi}{r}\right) - 
\exp\left(-\frac{i(c+1)k\pi}{r}\right)}
{\exp(\frac{ik\pi}{r}) - \exp(-\frac{ik\pi}{r})}
= \phi^{(r-2)/4} \frac{\sin\frac{(c+1)k\pi}{r}}{\sin(\frac{k\pi}{r})}.
$$
Hence, we obtain
$$
v_k \bullet_\tau v_k = \frac{\phi^{(r-2)/4}}{\sin(\frac{k\pi}{r})} \sum_{c=0}^{r-2}
\sin\left(\frac{(c+1)k\pi}{r}\right)  \hd_c = 
\phi^{(r-2)/4}\frac{\sqrt{r/2}}{\sin(\frac{k\pi}{r})} v_k.
$$

\paragraph{The quantum product $v_k \bullet_\tau v_l$ for $k\neq l$.}
The coefficient of $\hd_c$ in $v_k \bullet_\tau v_l$ is 
\begin{equation}\label{dvvd}
\frac{2 \phi^{(r-2)/4}}r \sum_{a,b} \sin\left( \frac{(a+1)k\pi}r \right) \sin\left( \frac{(b+1)l\pi}r \right)\, ,
\end{equation}
where the sum runs over the set shown in Figure~\ref{Fig:tetrahedron}. As before, 
we set $$p = a+b+2\, , \ \ \ q = a-b$$ and write \eqref{dvvd} as
{
\small{
\begin{multline*}
\frac{2\phi^{\frac{r-2}4}}r \sum_{p=c+2}^{2r-2-c} \sum_{q=-c}^c 
\sin\left( \frac{(p+q)k\pi}{2r} \right) \sin\left( \frac{(p-q)l\pi}{2r} \right)= \\
 \frac{\phi^{\frac{r-2}4}}r \sum_{p=c+2}^{2r-2-c} \sum_{q=-c}^c \left[
\cos\left( \frac{q(k+l)+p(k-l)}{2r}\pi \right)- \cos \left( \frac{p(k+l)+q(k-l)}{2r} \pi \right)
\right]\, ,
\end{multline*}
}}

\noindent where the sums run over only over integers $p$ and $q$
which have the same parity (mod 2) as $c$.
The result can be rewritten as
{\small{$$
\frac{\phi^{\frac{r-2}4}}r\, \sum_{p=c+2}^{2r-2-c} \sum_{q=-c}^c \Re \left[
\exp\left(\! \frac{q(k+l)+p(k-l)}{2r}i\pi \!\right)- \exp \left(\! \frac{p(k+l)+q(k-l)}{2r} i\pi \!\right)
\right],
$$}}
where $\Re$ denotes the real part as before. 
After further transformation, we obtain
{\small{
\begin{multline*}
\frac{\phi^{(r-2)/4}}r\,  \Re 
\sum_{p=c+2}^{2r-2-c} \sum_{q=-c}^c 
\exp\left( \frac{q(k+l)}{2r}i\pi \right) \exp\left( \frac{p(k-l)}{2r}i\pi \right)\\
- 
\frac{\phi^{(r-2)/4}}r\,  \Re 
\sum_{p=c+2}^{2r-2-c} \sum_{q=-c}^c
\exp \left( \frac{p(k+l)}{2r} i\pi \right) \exp \left( \frac{q(k-l)}{2r} i\pi \right)\, .
\end{multline*}}}
Consider the following four functions:
\begin{eqnarray*}
A =\sum_{q=-c}^c 
\exp \!\left( \frac{q(k+l)}{2r}i\pi \right)\ \ & \!\!=& 
\frac{\exp\left(\frac{(c+1)(k+l)i\pi}{2r}\right) - \exp\left(-\frac{(c+1)(k+l)i\pi}{2r}\right)}{\exp\left(\frac{(k+l)i\pi}{2r}\right) - \exp\!\left(-\frac{(k+l)i\pi}{2r}\right)}\, , \\
B =\sum_{p=c+2}^{2r-2-c} 
\exp\!\left( \frac{p(k-l)}{2r}i\pi \right)  &\!\!=& 
\frac
{\exp\left(\frac{(2r-1-c)(k-l)i\pi}{2r}\right) -
\exp\left(\frac{(c+1)(k-l)i\pi}{2r}\right)}
{\exp\left(\frac{(k-l)i\pi}{2r}\right) - 
\exp\left(-\frac{(k-l)i\pi}{2r}\right)}\\
& \!\! =& - \frac
{\exp\left(\frac{(c+1)(k-l)i\pi}{2r}\right) 
- \exp\left(-\frac{(c+1)(k-l)i\pi}{2r}\right)}
{\exp\left(\frac{(k-l)i\pi}{2r}\right) 
- \exp\left(-\frac{(k-l)i\pi}{2r}\right)}\, ,\\
C=\sum_{q=-c}^c 
\exp\!\left( \frac{q(k-l)}{2r}i\pi \right) \ \ &\!\!=& 
\frac{\exp\left(\frac{(c+1)(k-l)i\pi}{2r}\right) - \exp\left(-\frac{(c+1)(k-l)i\pi}{2r}\right)}
{\exp\left(\frac{(k-l)i\pi}{2r}\right) 
- \exp\left(-\frac{(k-l)i\pi}{2r}\right)}\, , \\
D=\sum_{p=c+2}^{2r-2-c} 
\exp\!\left( \frac{p(k+l)}{2r}i\pi \right)  &\!\!=& 
\frac
{\exp\left(\frac{(2r-1-c)(k+l)i\pi}{2r}\right) -
\exp\left(\frac{(c+1)(k+l)i\pi}{2r}\right)}
{\exp\left(\frac{(k+l)i\pi}{2r}\right) - 
\exp\left(-\frac{(k+l)i\pi}{2r}\right)}  \\
& \!\!= & - \frac
{\exp\left(\frac{(c+1)(k+l)i\pi}{2r}\right) 
- \exp\left(-\frac{(c+1)(k+l)i\pi}{2r}\right)}
{\exp\left(\frac{(k+l)i\pi}{2r}\right) 
- \exp\left(-\frac{(k+l)i\pi}{2r}\right)}\, .
\end{eqnarray*}
Since $AB - CD = 0$, we conclude
$$v_k \bullet_\tau v_l=0$$ for $k\neq l$.

\section{Tautological relations} \label{ttttt}
\subsection{Proof of Theorem \ref{Thm:TautRel1}}
\label{dddd4}
By Teleman's classification~\cite{Teleman}, the semisimple CohFT 
$\W^{r,\tau}$ for
$$\tau=(0,\ldots, 0, r\phi) \in V_r$$
defined in Section \ref{frrz}
is given by Givental's action \cite{Givental} of the $R$-matrix computed in Section~\ref{Ssec:Rmatrix1} on the topological field theory $\omega^{r,\tau}$ 
computed in Section~\ref{Ssec:TopFT1}. 
By the degree bound, 
\begin{multline*}
\deg
\Bigl[ p_{m*} \W^r_{g,n+m}(e_{a_1} \otimes \cdots \otimes e_{a_n}\otimes
\tau \otimes \cdots \otimes \tau) \Bigr] \\
 \leq \ 
\frac{(g-1)(r-2) + \sum a_i + m(r-2)}{r} -m \\
 = \ \D_{g,n}(a_1, \dots, a_n) - \frac{2m}r\, , 
{ \hspace{+154 pt} }
\end{multline*} 
$
\W^{r,\tau}_{g,n}(e_{a_1}\otimes \cdots \otimes e_{a_n})\in H^*(\oM_{g,n})
$
is a cohomology class with top degree 
$$
\D^r_{g,n}(a_1, \dots, a_n) = \frac{(g-1)(r-2)+\sum a_i}{r}
$$
equal to $W^r_{g,n}(a_1,\ldots,a_n)\in H^{2\D_{g,n}^r({\bf{a}})}(\oM_{g,n})$
and lower degree terms.

Hence, for any $d>\D^r_{g,n}(a_1,\ldots,a_n)$, the degree~$d$ part of the 
stable graph expression 
of $$\Omega^{r,\tau}=R.\omega^{r,\tau}$$ vanishes. The
proof of Theorem \ref{Thm:TautRel1} is complete. \qed

\vspace{10pt}
The proof of Theorem \ref{Thm:TautRel1} also yields an explicit
calculation of Witten's $r$-spin class.

\begin{theorem} \label{Thm:rWitten}
$W^r_{g,n}(a_1, \dots, a_n)$
equals the degree $\D^r_{g,n}(a_1,\ldots,a_n)$ part of the stable graph
expression of
$$\Omega^{r,\tau}= R.\omega^{r,\tau}\, $$
in $H^*(\oM_{g,n})$.
\end{theorem}

As a consequence of Theorem \ref{Thm:rWitten}, we see
\begin{equation}\label{fftt22}
 W^r_{g,n}(a_1, \dots, a_n)\in RH^*(\oM_{g,n})\, . 
\end{equation}
Since Theorem \ref{Thm:rWitten} and the implication \eqref{fftt22}
do  {\em not} concern relations, Janda's results do {\em not}
apply: we do {\em not} have a proof of the lifts to Chow of these
two statements.

\subsection{An example: $g=1$, $n=1$}
Let $r\geq 2$.
Let $g=1$, $n=1$, $a_1 = a\in \{0,\ldots, r-2\}$, $d=1$. We have
$$
\D^r_{1,1}(a) = \frac{a}r < 1\, .
$$
Thus, the degree~1 part of the stable graph expression of 
$$\Omega^{r,\tau}_{1,1}(e_a)=\phi^{(2a-r+2)/4} \, \Omega^{r,\tau}_{1,1}(\hd_a)$$
is a tautological relation for every $a$.
To write the relations, we will require the coefficient 
$$
\bb_{r,a} = \frac{(2a+2+r)(2a+2-r)}{16r^2}
$$
of $T$ in $\B_{r,a}$.  

There are exactly three stable graphs{\footnote{We refer the
reader to \cite[Section 2]{PaPiZv} for a review of Givental's action of
$R$ on $\omega^{r,\tau}$.}}
which contribute to the degree $1$ part of $\Omega_{1,1}^{r,\tau}(\hd_a)$: the graph $\Gamma_1$ with one genus~1 vertex and one marked leg, the graph $\Gamma_2$ with one genus~1 vertex, one marked leg, and one $\kappa$-leg, and the graph $\Gamma_3$ with one genus~0 vertex, a leg, and a loop.
\begin{center}
\ 
\begin{picture}(0,0)%
\includegraphics{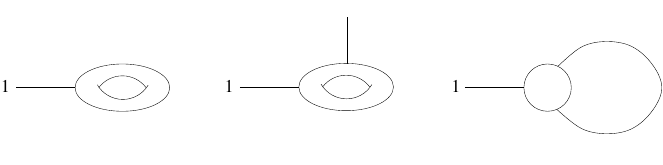}%
\end{picture}%
%
%
\setlength{\unitlength}{1579sp}%
\begingroup\makeatletter\ifx\SetFigFont\undefined%
\gdef\SetFigFont#1#2#3#4#5{%
  \reset@font\fontsize{#1}{#2pt}%
  \fontfamily{#3}\fontseries{#4}\fontshape{#5}%
  \selectfont}%
\fi\endgroup%
\begin{picture}(13248,3281)(151,-2591)
\put(6787,386){\makebox(0,0)[lb]{\smash{{\SetFigFont{10}{12.0}{\rmdefault}{\mddefault}{\updefault}{\color[rgb]{0,0,0}$\kappa$}%
}}}}
\put(1655,-2468){\makebox(0,0)[lb]{\smash{{\SetFigFont{10}{12.0}{\rmdefault}{\mddefault}{\updefault}{\color[rgb]{0,0,0}$\Gamma_1$}%
}}}}
\put(6379,-2468){\makebox(0,0)[lb]{\smash{{\SetFigFont{10}{12.0}{\rmdefault}{\mddefault}{\updefault}{\color[rgb]{0,0,0}$\Gamma_2$}%
}}}}
\put(10867,-2468){\makebox(0,0)[lb]{\smash{{\SetFigFont{10}{12.0}{\rmdefault}{\mddefault}{\updefault}{\color[rgb]{0,0,0}$\Gamma_3$}%
}}}}
\end{picture}%

\end{center}

\begin{enumerate}
\item[$\bullet$]
$\Gamma_1$ contributes
$
\phi^{-r/2} \; \bb_{r,a} \, \omega^{r,\tau}_{1,1}(\hd_{r-2-a}) \;  \psi_1\, ,
$
\item[$\bullet$]
$\Gamma_2$ contributes
$
- \phi^{-r/2}\; \bb_{r,0}  \, \omega^{r,\tau}_{1,2}(\hd_a \otimes \hd_{r-2}) \; \kappa_1\, ,
$
\item[$\bullet$]
$\Gamma_3$ contributes
$
- \phi^{-r/2}\; \sum_{a'=0}^{r-2}
\bb_{r,r-2-a'} \, \omega^{r,\tau}_{0,3}(\hd_a \otimes \hd_{a'} \otimes \hd_{a'}) \; \delta\, .
$
\end{enumerate}
Here,  $\delta$ is the class Poincar\'e dual to the boundary. 
The factor $\phi^{-r/2}$ comes from the series $\B_{r,a}(\phi^{-r/2} z)$ in the $R$-matrix.

Using Proposition~\ref{Prop:QuantumProd1}, we obtain the evaluations:
\begin{eqnarray*}
\omega^{r,\tau}_{1,1}(\hd_{r-2-a}) &=& (a+1) \phi^{(r-2)/4} \delta^{\rm even}_{r-a}\, ,\\
\omega^{r,\tau}_{1,2}(\hd_a \otimes \hd_{r-2}) &=& (a+1) \phi^{(r-2)/4} \delta^{\rm even}_{r-a}\, , 
\end{eqnarray*}
$$
\omega^{r,\tau}_{0,3}(\hd_a \otimes \hd_{a'} \otimes \hd_{a'}) = 
\left|
\begin{array}{cl}
\phi^{(r-2)/4} & \mbox{ if }  r-a \mbox{ is even and } \frac{r-2-a}2 \leq a' \leq \frac{r-2+a}2,\\
0 & \mbox{otherwise}.
\end{array}
\right.
$$
Hence, we have
\begin{eqnarray*}
\phi^{-\frac{r}2}\, \bb_{r,a} \, \omega^{r,\tau}_{1,1}(\hd_{r-2-a}) &\!=\!&
\phi^{-\frac{r+2}4} \, \frac{(2a+2+r)(2a+2-r)(a+1)}{16r^2}  \delta^{\rm even}_{r-a}\, , \\
- \phi^{-\frac{r}2}\, \bb_{r,0}  \, \omega^{r,\tau}_{1,2}(\hd_a \otimes \hd_{r-2})
&\! =\! &
\phi^{-\frac{r+2}4}\; \frac{(r+2)(r-2)(a+1)}{16r^2}  \delta^{\rm even}_{r-a}\, ,
\end{eqnarray*}
\begin{multline*}
-\phi^{-\frac{r}2}\; \sum_{a'=0}^{r-2}
\bb_{r,r-2-a'} \omega^{r,\tau}_{0,3}(\hd_a \otimes \hd_{a'} \otimes \hd_{a'})\\ 
=
-\phi^{-\frac{r+2}4}\; \delta^{\rm even}_{r-a}
\sum_{a' = (r-2-a)/2}^{(r-2+a)/2} \frac{(2a+2-r)(2a+2-3r)}{16r^2} \\
= -\phi^{-\frac{r+2}4}\;\frac{a(a+1)(a+2)}{48 r^2} \delta^{\rm even}_{r-a}\, .
\end{multline*}
After dividing by the common factor 
$\displaystyle -\phi^{-(r+2)/4}\, \frac{(a+1)}{16 r^2}$, we obtain the following statement.

\begin{proposition}\label{xxww}
For $a\in \{0,\ldots, r-2\}$ of the same parity as $r$, we have
$$
(r-2a-2)(r+2a+2) \cdot \psi_1 - (r-2)(r+2) \cdot \kappa_1 + \frac{a(a+2)}3 \cdot  \delta = 0 \in H^2(\oM_{1,1}).
$$
\end{proposition}

After regrouping the terms, we 
write the relation of Proposition \ref{xxww} as
$$(r^2- 4)(\psi_1 - \kappa_1) + (2a+a^2) \left(\frac{\delta}{3} - 4 
\psi_1\right)  = 0\, ,$$ 
which can only be satisfied for all the possible choices of $a$ and $r$  if
$$
\frac{\delta}{12} = \kappa_1 = \psi_1.
$$

\subsection{Restriction to $\cM_{g,n}$}\label{Subsec:Mg}
The tautological relations of Theorem \ref{Thm:TautRel1} become much simpler when restricted to the interior $\cM_{g,n} \subset \oM_{g,n}$ as the graph sum is 
then reduced to a single term. In order to prove Theorem \ref{Cor:Mg},
 we will alter the relations slightly before restricting to the interior.

Let $\sigma, a_1, \ldots, a_n$ and $d$ be as in the statement of 
Theorem \ref{Cor:Mg}, so
\[
rd > (r-2)(g-1) + |\sigma| + \sum_{i=1}^n a_i\, .
\]
Since $a_i$ and $\sigma_j$ avoid $r-1$ mod $r$, we can write
$$a_i = b_i + rc_i\, , \ \  \ \ \ \ \sigma_j = b_{n+j} + rc_{n+j}$$
with $0 \le b_i \le r-2$ for all $1\leq i \leq n+\ell(\sigma)$.
For $C = \sum_{i=1}^{n+\ell(\sigma)} c_i$, we have 
\[
r(d-C) > (r-2)(g-1) + \sum_{i=1}^{n+\ell(\sigma)}b_i\, .
\]
By Theorem~\ref{Thm:TautRel1} the degree $d-C$ part of
\[
\Omega^{r,\tau}_{g,n+\ell(\sigma)}(\hd_{b_1} \otimes \cdots \otimes \hd_{b_{n+\ell(\sigma)}})
\]
yields a tautological relation on $\oM_{g,n+\ell(\sigma)}$ which
we write as
\begin{equation}\label{gtt12}
 X = 0\in H^{2(d-C)}(\oM_{g,n+\ell(\sigma)})\, .
\end{equation}
Push-forward yields a
tautological relation of degree $d$ on $\oM_{g,n}$,
\begin{equation}\label{nnwww}
p_{\ell(\sigma)*} \left(
\prod_{j=1}^{\ell(\sigma)}\psi_{n+j}^{c_{n+j}+1}\cdot X\right) \, =\, 0\ \in \ H^{2d}(\oM_{g,n})\, .
\end{equation}
We restrict \eqref{nnwww} to the interior to obtain
 a tautological relation of degree $d$ on $\cM_{g,n}$.

The only stable graph for $\oM_{g,n+\ell(\sigma)}$ which contributes to the 
relation \eqref{gtt12} is the principal graph $\Gamma_\bullet$ with a single vertex
and no edges. 
All other strata classes are either annihilated by multiplying by 
$\prod_j\psi_{n+j}$ or  remain supported on the boundary after push-forward 
by the forgetful map to $\oM_{g,n}$.

All the factors in the contribution of $\Gamma_\bullet$
 match up exactly with the formula given in Theorem
\ref{Cor:Mg} except for the values of the topological field theory 
$\omega^{r,\tau}$ which are absent. 
A leg (or $\kappa$-leg) assigned vector $\hd_a$ produces the series $\B_{r,a}$ by combining the two nonzero entries in column $a$ of the matrix $R^{-1}$. 
The topological field theory value which appears is
\[
\omega^{r,\tau}_{g,n+\ell(\sigma)+m}(\hd_{\, \widehat{b}_1},\ldots,\hd_{\, \widehat{b}_{n+\ell(\sigma)}},
\hd_{\, \widehat{0}},\ldots,\hd_{\, \widehat{0}}),
\]
where $\widehat{x}$ represents either $x$ or $r-2-x$ depending on whether the diagonal or antidiagonal entry in $R^{-1}$ was 
chosen.
The number of times an antidiagonal{\footnote{The parity
condition here is because the antidiagonal entries are constructed from
{\em odd} functions.}} entry is chosen is congruent mod $2$ to the degree $d-C$. After specialization to $\phi=1$, the easily checked identity
\[
\omega^{r,\tau}_{g,s+1}(\hd_{x_1},\ldots,\hd_{x_s},\hd_{r-2}) = 
\omega^{r,\tau}_{g,s}(\hd_{x_1},\ldots,\hd_{x_{s-1}},\hd_{r-2-x_s}),
\]
then implies that the topological field theory value does not depend on any of the choices made (and can be divided out by the nonvanishing result below).

The parity condition in the statement of Theorem \ref{Cor:Mg},
\begin{equation*}
rd \equiv (r-2)(g-1) + |\sigma| + \sum_{i=1}^n a_i \, \mod{2}\, ,
\end{equation*}
together with parity constraint on the number of times an
antidiagonal is chosen, implies
$$\sum_{i=1}^{n+\ell(\sigma)} \widehat{b}_j \equiv r(g-1) \mod 2\, .$$

\begin{lemma}\label{Lem:nonzero}
Let $g \ge 1$, $s\ge 0$ with $2g-2+s > 0$. Let $x_1,\ldots,x_s \in 
\{0,\ldots,r-2\}$ satisfy
\[
\sum_{i=1}^sx_i \equiv r(g-1) \mod{2}.
\]
Then, $\omega^{r,\tau}_{g,s}(x_1,\ldots,x_s)\neq 0$.
\end{lemma}

\begin{proof}
We will use the formula for $\omega^{r,\tau}_{0,3}$ given by Proposition~\ref{Prop:QP} to induct on $g$ and $n$. In fact, the argument will prove 
 $$\omega^{r,\tau}_{g,s}(\hd_{x_1},\ldots,\hd_{x_s})> 0\, $$
for $\phi > 0$. 

Our base case is $\omega^{r,\tau}_{1,1}(x_1)$ with $x_1$ even. 
After applying the splitting axiom of the CohFT, we obtain
\[
\omega^{r,\tau}_{1,1}(\hd_{x_1}) = \sum_{x=0}^{r-2}
\omega^{r,\tau}_{0,3}(\hd_{x_1},\hd_{y},\hd_{r-2-y})
\]
and all of the terms are non-negative by Proposition~\ref{Prop:QP}.
For $y=\frac{x_1}{2}$, the point $(x_1,y,y)$ lies in the tetrahedron described 
in Proposition \ref{Prop:QP}, 
$$\omega^{r,\tau}_{0,3}\left(\hd_{x_1},\hd_{\frac{x_1}{2}},
\hd_{r-2-\frac{x_1}{2}}\right)>0\, .$$

Next, we prove the case $g=1$, $n>1$ by induction on $n$. We have 
\[
\omega^{r,\tau}_{1,n}(\hd_{a_1},\ldots,\hd_{a_{n-1}},\hd_{a_n}) = \sum_{x=0}^{r-2}\omega^{r,\tau}_{1,n-1}(\hd_{a_1},\ldots,\hd_{a_{n-2}},\hd_{x})\cdot 
\omega_{0,3}(\hd_{a_{n-1}},\hd_{a_n},\hd_{r-2-x}),
\]
and we may assume $a_{n-1}\equiv a_n \mod 2$.
As before, non-negativity means we  need only  find a single value of $x$ such that $\omega_{0,3}(a_{n-1},a_n,r-2-x)$ is nonzero. 
The tetrahedron constraints are satisfied for $x = |a_n-a_{n-1}|$.

Finally, we treat the case $g>1$ by induction on $g$. We have 
\[
\omega^{r,\tau}_{g,n}(\hd_{a_1},\ldots,\hd_{a_n}) = 
\sum_{x=0}^{r-2}\omega^{r,\tau}_{g-1,n+2}(\hd_{a_1},\ldots,\hd_{a_n},\hd_{x},
\hd_{r-2-x})
\]
and all of the terms on the right are positive by the inductive hypothesis (since the parity condition is preserved).
\end{proof}

\subsection{Proof of Theorem~\ref{Thm:betti}}\label{Subsec:betti}
We will now use the relations of Theorem \ref{Cor:Mg}
with $r=4$ to bound the Betti numbers of the tautological ring of 
$\cM_g$.{\footnote{The analysis here was completed by A. Pixton
before  our study of $r$-spin relations started and appears 
in \cite{PiThesis}. Since there is no published reference (and for the
convenience of the reader) we have included the short argument here.
Several aspects are parallel to the 
linear algebra required in \cite{kap}.}} 

In the case $r=4$ and $n=0$, the relations of Corollary~\ref{Cor:Mg} are 
parameterized by partitions $\sigma$ with no parts congruent to $3$ mod $4$ and positive integers $d$ satisfying
\[
4d > 2(g-1) + |\sigma| \ \ \ {\text{and}}\ \ \ 
|\sigma| \equiv 0 \pmod{2}.
\]

We discard the relations coming from partitions $\sigma$ containing an odd part and then halve all parts of $\sigma$. The remaining relations are then 
simply indexed by partitions $\sigma$ and positive integers $d$ satisfying
\[
2d \ge g + |\sigma|.
\]

For
$
D_s(T) = \B_{4,2s}(T)
$,
the relations of Theorem \ref{Cor:Mg} are obtained
by taking the degree $d$ part of 
\begin{equation}\label{Eq:FZ2}
\sum_{m\ge 0}\frac{1}{m!}\, p_{m+\ell(\sigma)*}\, \prod_{j=1}^{\ell(\sigma)}(TD_{\sigma_j})(\psi_j)\prod_{k=1}^m(T-TD_0)(\psi_{\ell(\sigma)+k})\, .
\end{equation}
We will alter the definition of $D_1$ to kill the  constant term:
\[
D_1(T) = \B_{4,2}(T) - \B_{4,0}(T).
\]
A straightforward check shows the span of the relations \eqref{Eq:FZ2}
is unchanged by the new definition of $D_1$.

\vspace{10pt}
\noindent{\em Proof of Theorem~\ref{Thm:betti}.} 
The push-forward kappa polynomials
\begin{equation}\label{zzqq}
p_{{\ell(\tau)}*}\, \prod_{i=1}^{\ell(\tau)}\psi_i^{\tau_i+1}\, , \ \ \ \ \ \ \
p_{\ell(\tau)}: \oM_{g,\ell(\tau)} \rightarrow \oM_{g}\, ,
\end{equation}
where $\tau$ is a partition of $d$, form a basis for the vector space of (formal) kappa polynomials of degree $d$. We will use the push-forward basis 
\eqref{zzqq} to obtain a lower bound for the rank of the relations given in \eqref{Eq:FZ2}.

Given any two partitions $\sigma$ and $\tau$, let 
$\mathsf{K}(\sigma,\tau)$ be the coefficient of the push-forward kappa polynomial corresponding to $\tau$ in \eqref{Eq:FZ2}. Define a matrix 
$\mathsf{M}$ with rows and columns indexed by partitions of $d$ by 
 $$\mathsf{M}_{\sigma\tau} = \mathsf{K}(\sigma_{-},\tau)\, ,$$ 
where $\sigma_{-}$ is the partition formed by reducing each part of $\sigma$ by 1 and discarding the parts of size 0.
By Proposition \ref{qq99qq} below,  $\mathsf{M}$ is invertible.

The invertibility of $\mathsf{M}$ implies Theorem \ref{Thm:betti}
by the following argument.
If $|\sigma| = d$ and $\ell(\sigma) \ge g-d$, then 
$$|\sigma_{-}| \le d - (g-d) = 2d-g\, , $$ so
 the row corresponding to $\sigma$ actually contains 
the coefficients of a relation obtained from Theorem \ref{Cor:Mg}.
Since $\mathsf{M}$ is invertible, 
all such  relations are linearly independent, 
so the quotient of the space of degree $d$ kappa polynomials 
by such relations has dimension at most the number of partitions of $d$ of length at most $g-1-d$, as desired. \qed
\vspace{10pt}

\begin{proposition} The matrix $\mathsf{M}$ is invertible. \label{qq99qq}
\end{proposition}

\begin{proof} We will show the
invertibility of $\mathsf{M}$ by constructing another matrix $\mathsf{A}$ 
of the same size and checking that the product $\mathsf{MA}$ is 
upper-triangular (with nonvanishing diagonal entries)
with respect to any ordering of the partitions of $d$ which
 places partitions containing more parts of size 1 after partitions containing 
fewer parts of size 1.

First, we compute the coefficient $\mathsf{K}(\sigma,\tau)$ as a sum over injections from the set of parts of $\sigma$ to the set of parts of $\tau$ describing which factors in \eqref{Eq:FZ2} produce which psi powers.
 We write such an injection as 
$$\phi:\sigma\hookrightarrow\tau\, .$$
The parts of $\tau$ which are not in the image of $\phi$ are produced by the factors involving $D_0$. The result is
\begin{equation}\label{eq:KKK}
\mathsf{K}(\sigma,\tau) = \frac{(-1)^{\ell(\tau)-\ell(\sigma)}}{|\Aut(\tau)|}\sum_{\phi\,:\,\sigma\,\hookrightarrow\,\tau}\prod_{i\,\stackrel{\phi}{\mapsto}j}[D_i]_{T^j}\prod_{j\in(\tau\setminus\phi(\sigma))}[D_0]_{T^j}\, .
\end{equation}

We define the matrix $\mathsf{A}$ as follows.
For any partitions $\tau$ and $\mu$ of the same size, 
\[
\mathsf{A}_{\tau,\mu} = \sum_{\substack{\psi\,:\,\tau\to\mu \\ \text{refinement}}}\frac{|\Aut(\tau)|}{\prod_{k\in\mu}|\Aut(\psi^{-1}(k))|} \prod_{k\in\mu} \big(\ell(\psi^{-1}(k)) + 2k + 1\big)! \prod_{j\in\tau}\frac{1}{(2j+1)!!}\, ,
\]
where the sum runs over all partition refinements 
$$\psi:\tau\to\mu\, ,$$ 
functions from the set of parts of $\tau$ to the set of parts of $\mu$ 
such that the preimage of each part $k$ of $\mu$ is a partition of $k$.

We factor the sums appearing in the entries of the product matrix 
$\mathsf{MA}$:
\begin{equation}\label{eq:KA}
\sum_{\tau} \mathsf{K}(\sigma_-,\tau)
\mathsf{A}_{\tau,\mu} = \sum_{\xi:\sigma_-\to\mu}\prod_{\substack{k\in\mu \\ \sigma' = \xi^{-1}(k)}}\left(\sum_{\tau'}\mathsf{K}(\sigma',\tau')\mathsf{A}_{\tau',(k)}\right)\, ,
\end{equation}
where $\xi:\sigma_-\to\mu$ is a function from the set of parts of $\sigma_-$ 
to the set of parts of $\mu$.

In order to understand \eqref{eq:KA},  we must study the sum
\begin{equation}\label{jjj2}
\sum_{\tau}\mathsf{K}(\sigma_-,\tau)\mathsf{A}_{\tau,(k)}\, .
\end{equation}
After  expanding \eqref{jjj2} via formula \eqref{eq:KKK} 
for $\mathsf{K}(\sigma_-,\tau)$ and the definition of $\mathsf{A}_{\tau,(k)}$, 
the result is
\begin{multline}\label{papapa}
(-1)^{\ell(\sigma)}\sum_{\tau\, \vdash k}\, \sum_{\phi\,:\,\sigma_-\,\hookrightarrow\,\tau}\frac{(-1)^{\ell(\tau)}(\ell(\tau)+2k+1)!}{|\Aut(\tau)|}\\
\cdot \prod_{j\in\tau}\frac{1}{(2j+1)!!}\prod_{i\,\stackrel{\phi}{\mapsto} j}[D_i]_{T^j}\prod_{j\in(\tau\setminus\phi(\sigma_-))}[D_0]_{T^j}\, ,
\end{multline}
where the first sum is over all partitions $\tau$ of $k$.
Next, we include formal variable $t$ to keep track of the size of $\tau$ 
 and factor based on the values of the images of the parts of $\sigma_-$ under $\phi$ via the  series
\[
\widehat{D}_i(t) = \sum_{j\ge 1}[D_i]_{T^j}\frac{t^{j+\frac{1}{2}}}{(2j+1)!!}\, .
\]
After removing nonzero scaling factors, we rewrite \eqref{papapa} as 
\begin{equation}\label{eq:Dcoeff}
\left[\widehat{D}_0^{-2k-2}\prod_{i\in\sigma_-}\frac{\widehat{D}_i}{\widehat{D}_0}\right]_{t^{-1}}\, .
\end{equation}

Up to a triangular change of basis in the $\widehat{D}_i$, we have
\[
\widehat{D}_i = \sin\left(\frac{2i+1}{2}\sin^{-1}(\sqrt{t})\right)\, .
\]
We define
\[
\theta = \frac{1}{2}\sin^{-1}(\sqrt{t})
\]
and check the following two properties:
\[
\left[\frac{1}{\sin^4\theta}\right]_{t^{-1}} \ne 0\, ,
\]
but
\[
\left[\frac{1}{\sin^e\theta}\right]_{t^{-1}} = 0
\]
for every even $e\ge 6$.

Using these facts to compute \eqref{eq:Dcoeff}, we conclude 
\[
\sum_{\tau}\mathsf{K}(\sigma_-,\tau)\mathsf{A}_{\tau,(k)} = 0
\]
whenever $|\sigma_-| < k-1$, and 
\begin{equation}\label{fdd3}
\sum_{\tau}\mathsf{K}((k-1),\tau)\mathsf{A}_{\tau,(k)} \neq 0\, .
\end{equation}

We now return to the 
matrix $\mathsf{MA}$. Suppose that $\sigma$ and $\tau$ are partitions of $d$ containing $S$ and $T$ parts equal to 1 respectively and
 $$\mathsf{MA}_{\sigma\tau}\ne 0\,. $$ 
By the identity \eqref{eq:KA} and the analysis above, 
there exists a function 
$$\xi:\sigma_{-}\to\tau$$ for which
 the preimage of each part $k\in\tau$ is a partition of size at least $k-1$. 
Thus, we have
\begin{equation}\label{kkww22}
|\sigma| - \ell(\sigma) = |\sigma_{-}| \ge |\tau| - \ell(\tau)\, .
\end{equation}
Since $|\sigma| = |\tau| = d$, we see  $\ell(\tau) \ge \ell(\sigma)$.
Moreover, comparing lengths of the partitions, we obtain
\begin{equation}\label{kkww33}
\ell(\sigma) - S = \ell(\sigma_{-}) \ge \ell(\tau) - T\, .
\end{equation}
Adding the inequalities \eqref{kkww22} and
\eqref{kkww33}, we conclude $S\leq T$. 

If $S=T$, then the $\xi$-preimage of each part $k\in\tau$ must be of size $k-1$ and have length one if $k>1$, which implies $\tau = \sigma$. Thus,
 the matrix $\mathsf{MA}$ is triangular. The nonvanishing of the diagonal entries follows from the nonvanishing \eqref{fdd3}. \end{proof}

\section{The semisimple point: 
$\widetilde{\tau} = (0, r \phi, 0,\ldots, 0)$}

\label{pfpfpf}
\subsection{Another shift}
The shift along the second basis vector $e_{1}\in V_r$ also yields
a semisimple CohFT with attractive properties.
The associated topological field theory $\omega^{r,\widetilde{\tau}}$
is very simple, much simpler than $\omega^{r,\tau}$,
but the $R$-matrix is not as explicit. A basic polynomiality
property of Witten's $r$-spin class will be proven using
$\widetilde{\tau}$. 

\subsection{The quantum product}\label{xxx123}
Recall the notation for genus $0$ correlators,
$$
\blangle a_1, \dots, a_n \brangle^r =\int_{\oM_{0,n}} W^r_{0,n}(a_1, \dots, a_n)
\,\in \mathbb{Q}\,  ,
$$
discussed in Section \ref{Sec:sl2}.
 As before, we will
often drop
the superscript $r$.

\begin{proposition}\label{lqq}
We have the evaluations
\begin{align*}
\blangle a,b,c\brangle & = 1 \quad \mbox{for } a+b+c = r-2\, ,\\
\blangle a,b,c,1 \brangle &= \frac1r \quad \mbox{for } a+b+c = 2r-3\, .
\end{align*}
All correlators involving an $a_i =1$ vanish whenever $n \geq 5$.
\end{proposition}

\begin{proof}  
The values of 3- and 4-point correlators are well-known (and
were stated in \eqref{fred} of Section \ref{wsc}).
The only nontrivial claim here is the vanishing for $n\geq 5$ which 
is a direct consequence of  Proposition \ref{Cor:vanishing}.
\end{proof}

By Proposition \ref{lqq},
the {\em quantum product} at $\widetilde{\tau}$ is given  by
$$
\d_a \bullet_{\widetilde{\tau}} \d_b = 
\left|
\begin{array}{lll}
\d_{a+b} & \mbox{if} & a+b \leq r-2\, ,\\
\phi \, \d_{a+b-r+1} & \mbox{if} & a+b \geq r-1\, .\\
\end{array}
\right.
$$
To simplify computations, we introduce a new frame{\footnote{The metric
is $\eta( \hhd_a,\hhd_b) = \phi^{-\frac{r-2}{r-1}}\delta_{a+b,r-2}$.}}   
$$
\hhd_a = \phi^{-a/(r-1)} \d_a.  
$$
The quantum multiplication then takes the form
$$
\hhd_a \bullet_{\widetilde{\tau}} \hhd_b = 
\left|
\begin{array}{lll}
\hhd_{a+b} & \mbox{if} & a+b \leq r-2\, ,\\
\hhd_{a+b-r+1} & \mbox{if} & a+b \geq r-1\, .\\
\end{array}
\right.
$$

\subsection{The topological field theory} \label{Ssec:TopFT2}

\begin{proposition} \label{Prop:TopFT2}
We have 
$$
\omega^{r,\widetilde{\tau}}_{g,n} (\hhd_{a_1} \otimes \cdots \otimes \hhd_{a_n})
= \phi^{(g-1)\frac{r-2}{r-1}}(r-1)^g \cdot \delta\, ,  
$$
where $\delta$ equals~$1$ if $g-1-\sum_{i=1}^n a_i$ is divisible by $r-1$ and~$0$ otherwise.
\end{proposition}

\begin{proof}
From Proposition~\ref{lqq} and the definition of $\hhd_a$ we get
$$
\omega^{r,\widetilde{\tau}}_{0,3} (\hhd_a \otimes \hhd_b \otimes \hhd_c)
= \left|
\begin{array}{ll}
\phi^{-\frac{r-2}{r-1}} & \mbox{ if } \; a+b+c = -1 \mod {r-1},\\
0 & \mbox{ else}.
\end{array}
\right.
$$ 
The topological field theory $\omega^{r,\widetilde{\tau}}_{g,n}$ for general $g$ and $n$  can be computed
by restricting the $\widetilde{\tau}$-shifted $r$-spin theory $\W^{r,\widetilde{\tau}}_{g,n}$ 
to $$[C]\in \oM_{g,n}\, ,$$
where $C$ is a completely degenerate curve with 
$2g-2+n$ rational components and $3g-3+n$ nodes.

The $3g-3+n$ nodes divide the $C$ into genus 0 components with 3 special points each. 
By the splitting axiom,
we must place insertions  $\{0, \dots, r-2 \}$ on every branch of every node
in a manner such that the following conditions are satisfied:
\begin{enumerate}
\item[(i)] the sum of the two insertions at each node equals $r-2$,
\item[(ii)] the sum of the three insertions on  each rational component of the curve plus 1
 is divisible by $r-1$.
\end{enumerate}
Conditions (i) and (ii) are impossible to satisfy if $g-1-\sum_{i=1}^n a_i$ is not divisible by $r-1$. 

If the divisibility condition is satisfied, we can first place an arbitrary insertion on 
a single branch of a node of every independent cycle of the dual graph of the curve. Then, the other 
insertions are uniquely determined. We find exactly $(r-1)^g$ possibilities.

Now each rational component contributes a factor of $\phi^{-\frac{r-2}{r-1}}$ and each node a factor $\phi^{\frac{r-2}{r-1}}$ (the inverse of the metric). Collecting all the factors we get
$$
\phi^{(g-1)\frac{r-2}{r-1}}(r-1)^g.
$$ 
\end{proof}

\subsection{Euler field and shifted degree} \label{Ssec:shift2}
The operator of quantum multiplication by the Euler field at $\widetilde{\tau}$,
$$E= (r-1) \phi^{\frac{r}{r-1}}\, \hhd_{1}\, ,$$ 
is  given in the frame $\{\hhd_a\}$ by the matrix
$$
\xi = 
(r-1) \phi^{\frac{r}{r-1}} 
\begin{pmatrix}
\nice 0 & \nice \cdots & \nice \cdots & \nice 0  & \nice 1\\
 \qquad  & \qquad & \qquad & \qquad & \qquad \\
\nice 1 & \nice 0 & &  & 0 \\
\\
\nice 0 & \nice 1 & \nice 0 &  & \nice 0 \\
\\
\nice \vdots & \nice \ddots & \nice \ddots  & \nice \ddots & \nice \vdots \\
\\
\nice 0 & \nice \cdots & \nice 0  & \nice 1 & \nice 0 
\end{pmatrix}
$$
or $\xi^0_{r-2} = \xi^{a+1}_a = (r-1) \phi^{\frac{r}{r-1}}$. 
In the same frame, the shifted degree operator is  
$$
\mu = \frac1{2r}
\begin{pmatrix}
\nice -(r-2) & \nice 0 & \nice \cdots & \nice \cdots & \nice 0  \\
 \qquad & \qquad  & \qquad  & \qquad & \qquad \\
\nice 0 & -(r-4) & \nice 0 &  & 0 \\
\\
\nice \vdots & \nice \ddots & \nice \ddots  & \nice \ddots & \nice \vdots \\
\\
\nice 0 &  & \nice 0  & r-4 & \nice 0 \\
\\
\nice 0  & \nice \cdots  & \nice \cdots & \nice 0 & \nice r-2 
\end{pmatrix}
$$
or $\mu_a^a = \frac{2a-r+2}{2r}$.

\subsection{The $R$-matrix} \label{Ssec:Rmatrix2}

Define the polynomials $P_m(r,a)$ by the following recursive procedure.  Let $P_0(r,a) = 1$. For $m \geq 1$, let
\begin{align} \label{Eq:DefP}
P_m(r,a) &= 
\frac12\sum_{b=1}^a (2mr-r-2b) P_{m-1}(r,b-1)\\ \nonumber
& - \frac1{4mr(r-1)}
\sum_{b=1}^{r-2} (r-1-b)(2mr-b)(2mr-r-2b) P_{m-1}(r,b-1)\, .
\end{align}

The second summation can be extended to $r-1$ instead 
of $r-2$ because of the presence of the factor $r-1-b$. 
The second sum is then easily seen to be divisible both by $r$ and by $r-1$, so  $P_m$ is indeed a polynomial. The first few values are:
\begin{align*}
P_0 & = 1, \\
P_1 &= \frac12 a(r-1-a) - \frac1{24}(2r-1)(r-2),\\
P_2 & = \frac18 a^4 -\frac1{12} a^3(5r-1) + \frac1{48} a^2 (20r^2-5r-4)
-\frac1{48} a (r-1)(6r^2+7r-2)\\
& \ \ \ + \frac1{1152}(2r-1)(r-2)(2r^2+19r+2)\, .
\end{align*}
There appears to be no closed formula for the polynomials $P_m$. However, 
we will present a  closed expression for $P_m(0,a)$  in Proposition~\ref{Prop:Pcstterm} in terms of Bernoulli polynomials.

\begin{lemma} \label{Lem:P}
The polynomials $P_m$ satisfy the relations
\begin{align}
\label{Eq:P1}
P_m(r,a) - P_m(r,a-1) & = \frac12 (2mr-r-2a) P_{m-1}(r,a-1) \\
\label{Eq:P2}
P_m(r,0) &= P_m(r,r-1)
\end{align}
and are the unique solutions to these equations with initial condition $P_0=1$.
\end{lemma}

\begin{proof} 
Given $P_{m-1}$, equation \eqref{Eq:P1} determines $P_m$ uniquely up to a polynomial in $r$ independent of~$a$. 
Equation~\eqref{Eq:P2} for $P_{m+1}$ then determines this polynomial in~$r$. 
The uniqueness statement is therefore established.

Equation~\eqref{Eq:P1} follows directly from the definition of $P_m$.
To show that the second equation is satisfied by $P_{m+1}$, a 
calculation is required. The definition of $P_m$ implies 
$$
P_{m+1}(r,r-1) - P_{m+1}(r,0) = \frac12  \sum_{a=0}^{r-2} (2mr+r-2a-2) P_m(r,a)\, ,
$$
where we have substituted $a=b-1$ for the summation variable. 
From the definition of $P_m$, we obtain
\begin{align} \label{Eq:sum}
\nonumber
P_{m+1}(r,r-1) - P_{m+1}(r,0) = \hspace{-10em}\\
 &\frac14 \sum_{a=0}^{r-2} \sum_{b=1}^a (2mr+r-2a-2)(2mr-r-2b) P_{m-1}(r,b-1)\\
 \nonumber
& - \frac1{8mr(r-1)} \sum_{a=0}^{r-2} (2mr+r-2a-2) \\
\nonumber
& \hspace{40pt} \times
\sum_{b=1}^{r-2} (r-1-b)(2mr-b)(2mr-r-2b) P_{m-1}(r,b-1)\, .
\end{align}
Using the evaluation
$$
\sum_{a=b}^{r-2} (2mr+r-2a-2) = (r-1-b)(2mr-b),
$$
we obtain
$$
\sum_{a=0}^{r-2} (2mr+r-2a-2) = 2mr(r-1)\, .
$$
In equation~\eqref{Eq:sum},
 we exchange the summation order in the first term and use the identities
 above. We obtain
\begin{align*}
P_{m+1}(r,r-1) - P_{m+1}(r,0) = \hspace{-10em}\\
 &\frac14 \sum_{b=1}^{r-2} (r-1-b)(2mr-b)(2mr-r-2b) P_{m-1}(r,b-1)\\
& - \frac{1}{8mr(r-1)} 2mr(r-1) \\
&\hspace{40pt} \times
\sum_{b=1}^{r-2} (r-1-b)(2mr-b)(2mr-r-2b) P_{m-1}(r,b-1)\, ,
\end{align*}
which clearly vanishes.
\end{proof}

\begin{proposition} \label{Prop:Rmatrix2}
The unique solution $R(z)=\sum\limits_{m=0}^\infty R_m z^m \in \End(V_r)[[z]]$ of the equations
$$
[R_{m+1},\xi] = (m+\mu) R_m
$$
with the initial condition $R_0=1$ has coefficients
$$
(R_m)^b_a = \left[-r(r-1)\phi^{\frac{r}{r-1}}\right]^{-m} \, P_m(r,r-2-b) \,, \ \ \ \  
 \mbox{if} \qquad  b+m = a \mod{r-1}
$$
and 0 otherwise. 
The inverse matrix $R^{-1}(z)$ has coefficients
$$
(R_m^{-1})^b_a = \left[r(r-1)\phi^{\frac{r}{r-1}}\right]^{-m} \, P_m(r,a)\,, \ \ \ \  
 \mbox{if} \qquad  b+m = a \mod{r-1}
 $$
 and 0 otherwise.
\end{proposition}

\begin{proof} The {\em uniqueness} of the solution follows from the semisimplicity of the Frobenius manifold $V_r$ at $\tau$ (proven
in Section \ref{xxx123}). Since $P_0=1$, 
the formula for $R_0$ yields the identity matrix. 
We must check that the formula 
for $R$ is indeed a solution of the recursion
$$
[R_{m+1},\xi] = (m+\mu) R_m.
$$
Explicitly, we must show
$$
(R_{m+1})_{a+1}^b \xi^{a+1}_a - \xi^b_{b-1} (R_{m+1})_a^{b-1}
= ( m + \mu_b^b) (R_m)_a^b,
$$
or, equivalently,
$$
(r-1)\phi^{\frac{r}{r-1}} 
 \left[(R_{m+1})_{a+1}^b -  (R_{m+1})_a^{b-1} \right]
= \frac{2mr -r+2b+2}{2r} (R_m)_a^b,
$$
where both $a+1$ and $b-1$ are understood modulo $r-1$.

The nonvanishing condition $b+m = a \pmod{r-1}$ is simultaneously satisfied or not satisfied in all three terms of the equality.
The formula for $R_m$ contains the factor
$\left[-r(r-1) \phi^{\frac{r}{r-1}}\right]^{-m}$. 
After using these two observations, we obtain
 the final form of the equality to be checked:
\begin{equation} \label{Eq:diff1}
P_{m+1}(r,r-1-b) - P_{m+1}(r,r-2-b) = \frac12 (2mr - r +2b+2) P_m(r,r-2-b)\, .
\end{equation}
Here, the argument $r-1-b$ of the polynomials should be taken modulo $r-1$. In other words, when $b=0$ the equality reads
\begin{equation} \label{Eq:diff2}
P_{m+1}(r,0) - P_{m+1}(r,r-2) = \frac12 (2mr - r +2) P_m(r,r-2)\, .
\end{equation}
We first prove \eqref{Eq:diff1}. After replacing $m$ with $m+1$ in \eqref{Eq:P1},
 we obtain
$$
P_{m+1}(r,a) - P_{m+1}(r,a-1) = \frac12 (2mr+r-2a) P_m(r,a-1)\, .
$$
After substituting $a=r-1-b$, we have
\begin{align*}
P_{m+1}(r,r-1-b) - P_{m+1}(r,r-2-b)  = \hspace*{-7em} \\
& = \frac12 \left( 2mr+r-2(r-1-b)\right) P_m(r,r-2-b)
\\
&= \frac12  (2mr - r +2b+2) P_m(r,r-2-b),
\end{align*}
which is exactly \eqref{Eq:diff1}. In particular, for $b=0$, we find
$$
P_{m+1}(r,r-1) - P_{m+1}(r,r-2)  = \frac12 (2mr - r +2) P_m(r,r-2).
$$
Equation \eqref{Eq:diff2} now follows from  
the equality $P_{m+1}(r,r-1) = P_{m+1}(r,0)$ of equation
\eqref{Eq:P2}.
\end{proof}

Let $\mathsf{W}^{r,\widetilde{\tau}}$ be the cohomological field theory given by the shift of Witten's $r$-spin class by the vector $\widetilde{\tau} = (0, r \phi, 0 , \dots, 0)$.
Define
$$\Omega^{r,\widetilde{\tau}}= R.\omega^{r,\widetilde{\tau}} $$ 
by the action of the $R$-matrix of Proposition~\ref{Prop:Rmatrix2} 
on the topological field theory $\omega^{r,\widetilde{\tau}}_{g,n}$ of Proposition~\ref{Prop:TopFT2}.
Using Teleman's classification and the dimension analysis of
Section \ref{dddd4},
we obtain the following result parallel to Theorem 
\ref{Thm:TautRel1} and Theorem \ref{Thm:rWitten} for the shift by 
${\tau} = (0,  \dots, 0, r\phi)$. As an outcome,
we obtain a second formula for Witten's $r$-spin class.

\begin{theorem} \label{Thm:SecondShift}
$W^r_{g,n}(a_1, \dots, a_n)$ equals
the part of  $\Omega^{r, \widetilde{\tau}}_{g,n}$ of degree 
$$
\D^r_{g,n}(a_1, \dots, a_n)=\frac{(r-2)(g-1) + \sum_{i=1}^n a_i}r\, 
$$
in $H^*(\overline{\mathcal{M}}_{g,n})$.
The parts of $\Omega^{r, \widetilde{\tau}}_{g,n}$ of degree
higher than $\D^r_{g,n}$ vanish. 
\end{theorem}

A more explicit formula for the cohomological
field theory $\Omega^{r, \widetilde{\tau}}$
 is provided in Proposition~\ref{Prop:SecondShift} of Section \ref{polypoly}.

\subsection{Proof of Theorem \ref{Thm:poly}}\label{polypoly}

We prove here Theorem~\ref{Thm:poly}: for $\sum_{i=1}^n a_i = 2g-2$, 
the rescaled Witten class
\begin{equation}\label{vttp}
\mathcal{W}_{g;a_1, \dots, a_n}(r) = r^{g-1} W^r_{g,n}(a_1, \dots, a_n)\ \in 
RH^{g-1}(\overline{\cM}_{g,n})
\end{equation}
is polynomial for sufficiently large values of~$r$. Our proof will show
the polynomial \eqref{vttp} is always divisible by $r-1$.

Theorems \ref{Thm:rWitten} and \ref{Thm:SecondShift} provide two formulas for Witten's $r$-spin class obtained
by shifting by $e_{r-2}$ and $e_1$ respectively. The formula
of Theorem \ref{Thm:rWitten} has an explicit $R$-matrix, but the topological field theory is not explicitly polynomial in~$r$. 
We will therefore use the formula
of Theorem \ref{Thm:SecondShift} which is explicitly polynomial. 
From now on, we will always assume that $r$ is sufficiently large.

Denote by $\mathsf{G}_{g,n+k}$ the set of stable graphs with $n+k$ legs.
The legs marked by $n+1$ to $n+k$, termed the {\em $\kappa$-legs}, will
correspond to marked points forgotten by a forgetful map.
The $\psi$ classes corresponding to the
$\kappa$-legs  push-forward to $\kappa$ classes.

Theorem~\ref{Thm:SecondShift} applied when  $\sum_{i=1}^n a_i = 2g-2$ expresses the rescaled Witten class $\mathcal{W}_{g;a_1, \dots, a_n}(r)$ as a sum over stable graphs with weightings.

\begin{definition} \label{Def:weighting}
Consider a stable graph $\Gamma$ of genus~$g$ with $n+k$ marked legs. 
A {\em weighting} $\mathbf{a}$ of\/ $\Gamma$ is a function on the set of
half-edges
$$\mathsf{H}(\Gamma) \rightarrow \{0, \dots, r-2 \}\, , \ \ \ h \mapsto a_h$$
 satisfying the following properties:
\begin{itemize}
\item[$\bullet$] if $h$ and $h'$ are the two half-edges of a single edge, 
then $a_h + a_{h'}=r-2$,
\item[$\bullet$] If $h$ 
corresponds to  the leg $i$ for $1 \leq i \leq n$, then\footnote{If $r-1$ were smaller than $a_i$, the weighting of the leg would be $a_i \mod r-1$, but here we assume $r$ is large.} $a_h=a_i$,
\item[$\bullet$] If $h$ is a $\kappa$-leg, then $a_h = 0$.
\end{itemize}
\end{definition}

To every vertex $v$ of a stable graph, we assign a formal variable $x_v$ satisfying $x_v^{r-1}=1$. For a polynomial $\Pi$ in variables $x_v$, we will denote by $\{ \Pi \}_x$ the term of degree~0 in all variables~$x_v$.

Given a stable graph $\Gamma$ with weighting, we assign 
%
to each edge $e\in \mathsf{E}(\Gamma)$ 
the {\em edge factor}
$$
\Delta(e) = \frac1{x^a y^b}\, \frac{1- \sum\limits_{m,\ell \geq 0} 
P_m(r,a) P_\ell(r,b) (x\psi')^m  (y\psi'')^\ell }{\psi' + \psi''}\, .
$$
Here, $a$ and $b$ are the weightings of the half-edges of~$e$, $\psi'$ and $\psi''$ are the corresponding cotangent line classes, and $x,y$ are the vertex variables 
corresponding to the vertices adjacent to the edge (if the edge is a loop,
 then $x=y$).

To each leg $i$ for $1 \leq i \leq n$ we assign the {\em leg factor}
$$
L(i) = \frac1{x^{a_i}_v} \sum_{m \geq 0} P_m(r,a_i)  (x_v\psi_i)^m,
$$
where $a_i$ is the weighting of the leg, $\psi_i$ is the 
cotangent line associated to the leg, and $x_v$ is the vertex variable of the vertex to which the leg is attached.

Finally, to each $\kappa$-leg $i$ for $n+1 \leq i \leq n+k$, we assign the {\em $\kappa$-factor}
$$
K(i) = - \psi_i \sum_{m \geq 1} P_m(r,0)  (x_v\psi_i)^m,
$$
where $\psi_i$ is the cotangent line class of the leg, and $x_v$ is the vertex variable of the vertex to which the leg is attached.

\begin{proposition} \label{Prop:SecondShift}
The class $\mathcal{W}_{g;a_1, \dots, a_n}(r)$ is  given by the degree $g-1$ part of the mixed degree cohomology class
$$
\sum_{k \geq 0} \sum_{
\substack{\Gamma \in \mathsf{G}_{g,n+k}\\
\mbox{\rm \scriptsize weightings} \; \mathbf{a}}
} \frac{(r-1)^{1-h^1(\Gamma)}}{|\Aut(\Gamma)|} p_*
\left\{  \prod_v x_v^{g_v-1} \prod_e \Delta(e) \prod_{i=1}^{n} L(i) 
\prod_{i=n+1}^{n+k} K(i) \right\}_x,
$$
where $p: \oM_\Gamma \to \oM_{g,n}$ is the natural boundary map forgetting the last $k$ marked points.
\end{proposition}

\begin{proof}
The formula is essentially a reformulation of the $R$-matrix action described in Theorem~\ref{Thm:SecondShift}. To simplify the computations, we set $\phi=1$.

The powers of  $x_v$ keep track of the remainders modulo~$r-1$. 
More precisely, the base vector $e_a$ corresponds to $x^{-a}$. 
The bi-vector $\eta^{ab}$ is then encoded by the expression $\frac{1}{x^ay^b}$
 with $a+b=r-2$. The matrix $R^{-1}_m$ takes $e_a$ to a multiple of $e_b$
 with $b=a-m \mod r-1$. Therefore the coefficients of $\psi^m$ in the formulas come with an $m$th power of the corresponding vertex variable. 
Finally, putting a factor $x_v^{g-1}$ on each vertex and taking the coefficient of $x_v^0$ allows one to encode the condition 
$$
g-1 - \sum a_i = 0 \mod r-1
$$
which appears in topological field theory $\omega^{r,\widetilde{\tau}}_{g,n}$. 

In order to take into account the rescaling by $r^{g-1}$, we have 
removed the factor $r^m$ from the denominator of the $m$th term $R_m$ of the $r$-matrix. Thus the degree $d$ part of the resulting mixed degree class is multiplied by $r^d$. In particular, the degree $g-1$ part, which corresponds to Witten's 
$r$-spin class, is multiplied by $r^{g-1}$ giving the rescaled class. 

Finally, we account for all the occurrences of $r-1$. There is a factor of $(r-1)^m$ in the denominator or $R_m$, leading, as above, 
to a global factor of $(r-1)^{1-g}$. There is also a factor $(r-1)^{g_v}$ in the topological field theory at the vertex~$v$. The latter yield  $(r-1)^{g-h^1(\Gamma)}$. 
After multiplying the two factors, we obtain $(r-1)^{1-h^1(\Gamma)}$.
\end{proof}

We will use  Proposition~\ref{Prop:SecondShift}
to prove the polynomiality assertion of Theorem \ref{Thm:poly}.
 However, we will temporarily remove the division by 
$\psi' + \psi''$ from the edge factor $\Delta$.
We will study the polynomiality in~$r$ of the formula
of Proposition~\ref{Prop:SecondShift} {\em without} the division by 
$\prod_{e\in \mathsf{E}(\Gamma)} (\psi'_e+\psi''_e)$.
For each stable graph $\Gamma$, we will prove the
polynomiality in $r$ of  the degree $$g-1+|\mathsf{E}(\Gamma)|$$ part of 
the formula of Proposition~\ref{Prop:SecondShift}
without  denominators, where $|\mathsf{E}(\Gamma)|$ is the number of edges.

The division by $\prod_{e\in \mathsf{E}(\Gamma)} (\psi'_e+\psi''_e)$ will be taken afterwards
by the following argument.
Consider the expression of Proposition~\ref{Prop:SecondShift} as an element of the strata algebra {\em not} quotiented by any tautological relations, not even the relations due to the degree of the cohomology class supported by $\oM_v$ being higher than the dimension of $\oM_v$ for some vertex~$v$. 
Then, $\psi'_e+\psi''_e$ is not a zero divisor in the
strata algebra. Division by $\prod_{e\in \mathsf{E}(\Gamma)} \psi'_e + \psi''_e$, when possible at all, 
is therefore uniquely defined and preserves the property of being a polynomial in~$r$. 


Let $\Gamma$ be a stable graph with $n+k$ legs. Let $\mathbf{m}$ be a function 
$$\mathbf{m}: \mathsf{H}(\Gamma) \rightarrow \mathbb{Z}_{\geq 0}\ , \ \ \ \ 
h \mapsto m_h$$
satisfying the constraint $\sum_{h\in \mathsf{H}(\Gamma)} m_h = g-1$ and the
condition
\begin{enumerate}
\item[$\bullet$]  if $h$ and $h'$ are the two half-edges of a single edge, then\\
 $(m_h,m_{h'}) \neq (0,0)$.
\end{enumerate}
Define the sum
$$
S_{\Gamma, \mathbf{m}} = \sum_{\mbox{\rm \scriptsize weightings} \; \mathbf{a}} p_*
\left\{ \prod_v x_v^{g_v-1} \prod_h P_{m_h}(r,a_h)x_{v(h)}^{m_h-a_h} \right\}_x\, .
$$
By Lemmas \ref{Lem:poly} and \ref{Lem:div} below, 
for $r$ large enough, {\em $S_{\Gamma, \mathbf{m}}$ is a polynomial in $r$ divisible by $(r-1)^{h^1(\Gamma)}$}.
By  writing the coefficients of the formula of Proposition~\ref{Prop:SecondShift}
without denominators in terms of the $S_{\Gamma,\mathbf{m}}$, we obtain
the polynomiality required for Theorem \ref{Thm:poly}.
Moreover, since the prefactor in Proposition~\ref{Prop:SecondShift}
for the rescaled Witten class is $(r-1)^{1-h^1(\Gamma)}$, we 
also conclude that the rescaled Witten class is a polynomial in $r$ 
{\em divisible by $r-1$}.



\begin{lemma} \label{Lem:poly}
The sum $S_{\Gamma, \mathbf{m}}$ is a polynomial in~$r$ for $r$ large enough.
\end{lemma}

\begin{proof}
The proof here follows closely Pixton's proof of polynomiality 
in \cite[Appendix]{JPPZ}.
We will use \cite[Proposition~A1]{JPPZ}, but with Pixton's $r$
replaced with $r-1$ (which we assume to be large enough). 
Let $\Gamma'$ be the graph obtained from $\Gamma$ by adding a vertex 
at the end of each leg and in the middle of each edge. 
Let $\mathsf{M}$ be the edge-vertex adjacency matrix of $\Gamma'$. 
The matrix $\mathsf{M}$ satisfies the assumptions of \cite[Proposition~A1]{JPPZ}. 
The vector $\mathsf{x}$ of  \cite[Proposition~A1]{JPPZ} 
assigns an integer $x_h$ to each edge of $\Gamma'$ or, in other words, 
to each half-edge $h$ of $\Gamma$. The vectors $\mathsf{a}$ and $\mathsf{b}$ 
of  \cite[Proposition~A1]{JPPZ}
assign an integer to each vertex of $\Gamma'$. 
The following table summarizes what these integers are 
for each vertex and what conditions the equation 
$$\mathsf{M} \mathsf{x} = \mathsf{a} + (r-1) \mathsf{b}$$ imposes. 
In fact, the latter conditions are equivalent to $\mathsf{x}$ being a weighting.

\bigskip

\noindent
\begingroup
\renewcommand*{\arraystretch}{2}
\begin{tabular}{|p{8em}|c|c|p{8em}|}
\hline
type of vertex of $\Gamma'$& $\mathsf{a}$ & $ \qquad \mathsf{b} \qquad$ & effect on $\mathsf{x}$\\
\hline
midpoint of edge $h$--$h'$ in $\Gamma$ & $r-2$ & $0$ & $x_h + x_{h'} = r-2$ \\
endpoint of leg $h$ in $\Gamma$ & $a_h$ & $0$ & $x_h = a_h$ \\
vertex $v$ of $\Gamma$ & $g_v-1+ \sum\limits_{\substack{h \mapsto v}}
m_h$ & $b_v$ & TopFT condition mod $r-1$ at $v$\\
\hline
\end{tabular}
\endgroup

\bigskip

The numbers $b_v$ in the table can take different values for different weightings.
However, 
for a given graph $\Gamma$ and a given choice of integers $m_h$, there are only {\em finitely} many 
 possible values $b_v$. Thus, the sum $S_{\Gamma, \mathbf{m}}$ over all 
weightings can be decomposed into a finite number of sums of the form of
\cite[Proposition~A1]{JPPZ}. Hence, by
\cite[Proposition~A1]{JPPZ},
 $S_{\Gamma, \mathbf{m}}$ is a polynomial in~$r$.
\end{proof}

\begin{lemma} \label{Lem:div}
The polynomial $S_{\Gamma, \mathbf{m}}$ is divisible by $(r-1)^{h^1(\Gamma)}$. 
\end{lemma}

\begin{proof}
Once again we follow Pixton's proof in \cite[Appendix]{JPPZ}. 
Let $\mathsf{Q}$ be a polynomial in $N$ variables with ($p$-integral) $\mathbb{Q}$-coefficients. 
According to \cite[Equation 33]{JPPZ}, the sum
\begin{equation}\label{fxxf2}
\sum_{0 \leq w_1, \dots, w_N \leq p} \mathsf{Q}(w_1, \dots, w_N)
\end{equation}
over the $N$-tuples satisfying $D$ integral linear equations mod $p$ is 
divisible by $p^{N-D}$ for every large enough prime $p$. 
In our case, the sum $S_{\Gamma, \mathbf{m}}$ has exactly the form 
\eqref{fxxf2} 
if we take $p=r-1$ to be prime. 

The number $N$ of variables is equal to the number $|\mathsf{E}(\Gamma)|$ of 
edges of the graph $\Gamma$. 
The number of mod~p linear equations is equal to $|\mathsf{V}(\Gamma)|-1$, 
where $V(\Gamma)$ is vertex set of $\Gamma$. 
Indeed, there is one mod $p$ condition per vertex, but one condition is redundant, 
since the sum of the conditions is equal to  
\begin{equation}\label{x99x}
2g-2-\sum a_i = 0 \mod p\, , 
\end{equation}
a condition that is automatically satisfied. 

We check the assertion that the sum of the vertex
condition yields \eqref{x99x} as follows. First, we sum
the vertex conditions
$$
\sum_{e\in \mathsf{E}(\Gamma)} w_e + (r-2-w_e) + \sum_{i=1}^n a_i - \sum_{h\in \mathsf{H}(\Gamma)} m_h = \sum_{v\in \mathsf{V}(\Gamma)} (g_v - 1)  \mod r-1\, .
$$
We rewrite the above as
$$
(r-2) |\mathsf{E}(\Gamma)| + \sum_{i=1}^n a_i - (g-1) = g-1-|\mathsf{E}(\Gamma)| \mod 
r-1\, ,
$$
or equivalently,
$$
(r-1) |\mathsf{E}(\Gamma)| + \sum_{i=1}^n a_i = 2g-2 \mod r-1\, , 
$$
which is exactly \eqref{x99x}.

Thus, by \cite[Appendix A.3]{JPPZ},
 $S_{\Gamma, \mathbf{m}}$ is divisible by 
$$(r-1)^{N-D} = (r-1)^{|\mathsf{E}(\Gamma)|-|\mathsf{V}(\Gamma)|+1} = (r-1)^{h^1(\Gamma)}$$
 for $r-1$ prime and large. Since we already know that
 $S_{\Gamma, \mathbf{m}}$ is a polynomial in $r$ for $r$ large enough, 
we conclude the polynomial is divisible by $(r-1)^{h^1(\Gamma)}$.
\end{proof}

\pagebreak

\begin{appendix}

\section{Holomorphic differentials}

$$\hspace{-10pt} {\text{\em by F. Janda, R. Pandharipande, A. Pixton, D. Zvonkine}}$$

\vspace{10pt}
\noindent{A.1\, \bf{Moduli space.}}
Let $g$ and $n$ be in the stable range $2g-2+n>0$, and  
let $$(a_1,\ldots,a_n)\,, \ \ \ \ \sum_{i=1}^n a_i=2g-2$$
be a partition with  $a_i\geq 0$ for all $i$.
 
We define the {\em moduli space of holomorphic differentials} 
as the closed substack
$$\HH_{g}(a_1, \dots, a_n)=\Big\{\, [C, {\mathsf{p}_1}, \ldots, 
{\mathsf{p}_n}]\in \cM_{g,n}\ \Big| \ {\mathcal{O}}_C\Big(\sum_{i=1}^n 
a_i {\mathsf{p}}_i\Big)=\omega_C \, \Big\}\, \subset \cM_{g,n}\, . $$
Since $\HH_{g}(a_1, \dots, a_n)$ is the locus of points
$$[C, {\mathsf{p}}_1, \ldots, {\mathsf{p}}_n]\in \cM_{g,n}$$ for which the evaluation map
$$H^0(C, \omega_C)\rightarrow H^0\bigl(C, \omega_{C | 
a_1{\mathsf{p}_1}+\cdots+a_n {\mathsf{p}_n}}\bigr)$$ is not injective, every component of $\HH_{g}(a_1, \dots, a_n)$ has dimension at least $2g-2+n$
in $\cM_{g,n}$ by degeneracy loci considerations \cite{Fulton}.
Polishchuk~\cite{Pol} has shown
that $\HH_{g}(a_1, \dots, a_n)$ is a \emph{nonsingular} substack of $\cM_{g,n}$ of
pure dimension $2g-2+n$. Hence, the Zariski closure
$$\HH_{g}(a_1, \dots, a_n) \subset \overline{\HH}_{g}(a_1, \dots, a_n) \subset \oM_{g,n}$$
defines a cycle class
$$[\overline{\HH}_{g}(a_1, \dots, a_n)] \in H^{2(g-1)}(\oM_{g,n})\, .$$
Our goal in the Appendix is to relate $[\overline{\HH}_{g}(a_1, \dots, a_n)]$
to a certain limit of Witten's $r$-spin classes. 

A compact  moduli space of twisted canonical divisors which
includes the moduli of holomorphic differentials $\HH_g(a_1,\ldots,a_n)$
is defined in \cite{FarP}.
A detailed study of the points of the closure
$\overline{\HH}_{g,n}(a_1, \dots, a_n)$
can be found in \cite{BCGGM}.
In \cite[Appendix]{FarP}, a conjecture determining 
$$[\overline{\HH}_{g}(a_1, \dots, a_n)] \in A^{g-1}(\oM_{g,n})\, $$
in terms the fundamental classes of the moduli spaces
of twisted canonical divisors and 
a formula of Pixton is presented.

The relationship
of the conjecture of \cite[Appendix]{FarP} to our conjecture
here is a direction for future study. 

\vspace{10pt}
\noindent{A.2\, \bf{The limit $r=0$.}}
Let $(a_1,\ldots,a_n)$ be a partition of $2g-2$
with non-negative parts (as in A.1). For 
$$r-2\geq \text{max}\{a_1,\ldots,a_n\}\, ,$$
Witten's $r$-spin class 
$W^r_{g,n}(a_1, \dots, a_n)$
is well-defined and of degree {\em independent} of $r$,
$${\mathsf{D}}^r_{g,n}(a_1,\ldots,a_n) = \frac{(r-2)(g-1) + \sum_{i=1}^n a_i}{r} =g-1\, .$$
By Theorem~\ref{Thm:poly},
 after scaling by $r^{g-1}$, 
$$
\mathcal{W}_{g; a_1, \dots, a_n}(r)=r^{g-1}\cdot W^r_{g,n}(a_1, \dots, a_n)\in RH^{g-1}(\oM_{g,n})
$$
is a {\em polynomial} in $r$ for all sufficiently large $r$.

\vspace{+10pt}
\begin{conjecture} \label{Conj:Hol}
We have
$$
(-1)^g \mathcal{W}_{g; a_1, \dots, a_n}(0) = [\overline{\HH}_{g}(a_1, \dots, a_n)]\in H^{2(g-1)}(\oM_{g,n}).
$$
\end{conjecture}

\vspace{8pt}
If the polynomiality of Theorem \ref{Thm:poly} were to hold in Chow
(which we expect), then Conjecture
\ref{Conj:Hol} could also be formulated in $A^{g-1}(\overline{\cM}_{g,n})$.
By Conjecture \ref{Conj:Hol}, 
$[\overline{\HH}_{g}(a_1, \dots, a_n)]$ is a tautological class --  a claim
which has been proven{\footnote{In fact, Sauvaget
proves $[\overline{\HH}_{g}(a_1, \dots, a_n)]$ is tautological in Chow.}}
by A. Sauvaget \cite{Sauv}.

\pagebreak
\vspace{10pt}
\noindent{A.3\, \bf{Evidence.}}

\paragraph{Genus~1.}
Witten's class $W^r_{1,n}(0, \dots, 0)$ has degree~0. To evaluate the class,
the topological field theory is enough 
(we do not need the $R$-matrix). 
A simple computation shows that Witten's class here is equal to $r-1$. 
Indeed,
\begin{eqnarray*}
W^r_{1,n}(0, \dots, 0) &=& \omega_{1,n}(0, \dots, 0)\\
&=& \sum_{a=0}^{r-2} \omega_{0,n+2}(0, \dots, 0, a,r-2-a)\\
&=& \sum_{a=0}^{r-2} 1\ =\ r-1\, .
\end{eqnarray*}
Thus $\mathcal{W}_{1; 0, \dots, 0}(r) = r-1$, 
$\mathcal{W}_{1;0, \dots, 0}(0) = -1$. 
According to the conjecture we have
$$
(-1)^g \cdot (-1) = 1 =  [\overline{\HH}_{1}(0, \dots, 0)]\in RH^0(\oM_{1,n}),
$$
which is, indeed, true.

\paragraph{Genus 2, $n=1$, $a_1=2$.}

The tautological space $RH^1(\oM_{2,1})$ is spanned by the classes
$\psi_1$, $\delta_{\rm sep}$, and $\delta_{\rm nonsep}$, where the indices sep and nonsep refer to the boundary divisors with a separating or a nonseparating node. The class $\kappa_1$ may be expressed as 
\begin{equation}\label{d122}
\kappa_1 = \psi_1 + \frac75 \delta_{\rm sep} + \frac15 \delta_{\rm nonsep}.
\end{equation}
Equation \eqref{d122} is obtained by  pulling back the relation on 
$\oM_2$ expressing $\kappa_1$ in terms of boundary divisors. 

Theorem~\ref{Thm:SecondShift} gives an expression for Witten's class as a linear combination of classes $\kappa_1$, $\psi_1$, $\delta_{\rm sep}$ and $\delta_{\rm nonsep}$ with the following coefficients:
\begingroup
\renewcommand*{\arraystretch}{2}
$$
\begin{array}{lccc}
\mbox{Class} & R\mbox{-matrix} & \mbox{TopFT} & \mbox{Coefficient} \\
\kappa_1  & -\frac{P_1(r,0)}{r(r-1)} &  (r-1)^2 &\ \frac{(r-1)(r-2)(2r-1)}{24r}\\
\psi_1 & \frac{P_1(r,2)}{r(r-1)} &  (r-1)^2 &\ -\frac{(r-1)(2r^2-29r+74)}{24r}\\
\delta_{\rm sep} & -\frac{P_1(r,1)}{r(r-1)} &  (r-1)^2 &\ \frac{(r-1)(r-2)(2r-13)}{24r} \\
\delta_{\rm nonsep} &- \sum\limits_{a=0}^{r-2} \frac{P_1(r,a)}{r(r-1)} &  r-1 &\ -\frac{(r-1)(r-2)}{24r}
\end{array}
$$
\endgroup

After multiplying Witten's class
 by $(-1)^g r^{g-1} = r$ and extracting the constant term in~$r$, we obtain
$$
\frac1{12} (-\kappa_1 + 37 \psi_1 - 13 \delta_{\rm sep} - \delta_{\rm nonsep})\, .
$$
We remove $\kappa_1$ using equation \eqref{d122}. Conjecture
\ref{Conj:Hol} predicts
$$
[\overline{\HH}_2(2)] = 
3 \psi_1 - \frac65 \delta_{\rm sep} - \frac1{10} \delta_{\rm nonsep} \ \in
RH^1(\overline{\mathcal{M}}_{2,1})\, .
$$
The result coincides with the well-known formula for the locus of the 
Weierstrass points, see \cite[Lemma~5]{bp}.

\paragraph{Genus~2, $n=2$, $a_1=a_2=1$.}
The tautological space $RH^1(\oM_{2,2})$ is spanned by six classes
$\psi_1$, $\psi_2$, $\alpha$, $\beta$, $\gamma$ and $\delta_{\rm nonsep}$ where
\begin{itemize}
\item
$\alpha$ is the locus of curves with a rational component carrying both markings and a genus 2 component,
\item
$\beta$ is the locus of curves with two elliptic components carrying one marking each,
\item
$\gamma$ is the locus of curves with two elliptic components one of which carries both markings and the other one no markings,
\item
$\delta_{\rm nonsep}$ is the locus of curves with a nonseparating node.
\end{itemize}
The class $\kappa_1$ is expressed as 
$$
\kappa_1 = \psi_1 + \psi_2 + \alpha +\frac75 \beta + \frac75 \gamma + \frac15 \delta_{\rm nonsep}\, 
$$
by pulling back the boundary relation on $\oM_2$. 

Theorem~\ref{Thm:SecondShift} gives an expression for Witten's class as a linear combination of classes $\kappa_1$, $\psi_1$, $\psi_2$, $\alpha$, $\beta$, $\gamma$ and $\delta_{\rm nonsep}$ with the following coefficients:
\begingroup
\renewcommand*{\arraystretch}{2}
$$
\begin{array}{lccc}
\mbox{Class} & R\mbox{-matrix} & \mbox{TopFT} & \mbox{Coefficient} \\
\kappa_1  & -\frac{P_1(r,0)}{r(r-1)} &  (r-1)^2 &\ \ \frac{(r-1)(r-2)(2r-1)}{24r}\\
\psi_1 & \frac{P_1(r,1)}{r(r-1)} &  (r-1)^2 &\ \ -\frac{(r-1)(r-2)(2r-13)}{24r}\\
\psi_2 & \frac{P_1(r,1)}{r(r-1)} &  (r-1)^2 &\ \ -\frac{(r-1)(r-2)(2r-13)}{24r}\\
\alpha & -\frac{P_1(r,2)}{r(r-1)} &  (r-1)^2 &\ \ \frac{(r-1)(2r^2-29r+74)}{24r} \\
\beta & -\frac{P_1(r,0)}{r(r-1)} &  (r-1)^2 &\ \ \frac{(r-1)(r-2)(2r-1)}{24r} \\
\gamma & -\frac{P_1(r,1)}{r(r-1)} &  (r-1)^2 &\ \ \frac{(r-1)(r-2)(2r-13)}{24r} \\
\delta_{\rm nonsep} &- \sum\limits_{a=0}^{r-2} \frac{P_1(r,a)}{r(r-1)} &  r-1 & \ -\frac{(r-1)(r-2)}{24r}
\end{array}
$$
\endgroup

After multiplying Witten's class by $(-1)^g r^{g-1} = r$ and extracting the constant term in~$r$, we obtain
$$
\frac1{12} (-\kappa_1 +13 \psi_1 +13 \psi_2  -
37 \alpha - \beta - 13 \gamma - \delta_{\rm nonsep})\, .
$$
We remove $\kappa_1$ using equation \eqref{d122}. Conjecture
\ref{Conj:Hol} predicts
%
$$ [\overline{\HH}_2(1,1)]
=\psi_1 +\psi_2 - 3 \alpha - \frac15 \beta - \frac65 \gamma - \frac1{10} \delta_{\rm nonsep}\ \in RH^1(\overline{\mathcal{M}}_{2,2})\, .
$$
The result coincides with the well-known formula for the locus of genus
2 curves with a pair of conjugate points, see \cite[Lemma~6]{bp}.

\vspace{10pt}
\noindent{A.4\, \bf{The constant term.}}
We will now present a more explicit approach to the constant term 
$$
(-1)^g \mathcal{W}_{g;a_1, \dots, a_n}(0) \in RH^{g-1}(\overline{\mathcal{M}}_{g,n})
$$
which, according to Conjecture~\ref{Conj:Hol}, equals  
$[\overline{\HH}_{g}(a_1, \dots, a_n)]$. 

We will use the shift along $e_1$ studied
in Section \ref{pfpfpf}. The corresponding
$R$-matrix involves a sequence of polynomials $P_m(r,a)$ for which we
know no closed formula. However, in Proposition \ref{Prop:Pcstterm} below,
we obtain a closed formula for the 
polynomials $P_m(0,a)$.

Let $B_m(x)$ be the Bernoulli polynomials defined by 
$$
\frac{t e^{xt}}{e^t-1}= \sum_{m=0}^\infty B_m(x) \frac{t^m}{m!}\, .
$$
Let $P_m(r,a)$ be the polynomials defined by equation~\eqref{Eq:DefP}. 

\begin{proposition} \label{Prop:Pcstterm}
We have
$$
\sum_{m \geq 0} z^m P_m(0,a) = \exp \left(  - \sum_{m \geq 1} z^m \frac{B_{m+1}(a+1)}{m(m+1)} \right).
$$
\end{proposition}

\begin{proof} Define polynomials $Q_m(a)$ by 
\begin{equation}\label{llw}
\sum_{m \geq 0} z^m Q_m(a) = \exp \left(  - \sum_{m \geq 1} z^m \frac{B_{m+1}(a+1)}{m(m+1)} \right)\, .
\end{equation}
We will show that the polynomials $Q_m(a)$ satisfy the mod~$r$ reduction of equations~\eqref{Eq:P1} and~\eqref{Eq:P2} of Lemma~\ref{Lem:P},
\begin{equation} \label{Eq:Q1}
Q_m(a) - Q_m(a-1) = -a Q_{m-1}(a-1)\, ,
\end{equation}
\begin{equation} \label{Eq:Q2}
Q_m(-1) = Q_m(0)\, .
\end{equation}
Together with the initial condition $Q_0=1$,
these properties determine the polynomials $Q_m$ uniquely and immediately imply  $$Q_m(a) = P_m(0,a)\,.$$

\noindent $\bullet$
 Property~\eqref{Eq:Q2} is implied by the simple equality 
$$B_{m+1}(0) = B_{m+1}(1)$$ for every~$m\geq 1$ and definition
\eqref{llw}. 

\vspace{8pt}
\noindent $\bullet$ Property~\eqref{Eq:Q1} follows from
a generating function calculation:
$$
\sum_{m \geq 0} z^m Q_m(a) - \sum_{m \geq 0} z^m Q_m(a-1)
= -a z \sum_{m \geq 0} z^m Q_m(a-1)\, 
$$
which is equivalent to
$$
\sum_{m \geq 0} z^m Q_m(a) 
= (1-a z) \sum_{m \geq 0} z^m Q_m(a-1)\, .
$$
After taking the logarithm
$$
\log \sum_{m \geq 0} z^m Q_m(a) 
= \log (1-a z) + \log \sum_{m \geq 0} z^m Q_m(a-1)\, ,
$$
we must show
$$
\sum_{m \geq 1} z^m \frac{B_{m+1}(a+1)}{m(m+1)}
= \sum_{m \geq 1} z^m \frac{a^m}{m}  + \sum_{m \geq 1} z^m 
\frac{B_{m+1}(a)}{m(m+1)}\, .
$$
The latter follows from
$$
B_{m+1}(a+1) = (m+1) z^m + B_{m+1}(a)  
$$
which is a well-known property of Bernoulli polynomials.
\end{proof}

Using Proposition \ref{Prop:Pcstterm},
we can derive a more explicit formula for 
the constant term
$$(-1)^g \mathcal{W}_{g,(a_1, \dots, a_n)} (0)$$ as a sum over stable graphs.
We will require here only stable graphs with exactly $n$ legs.
We will replace the push-forward of the $\psi$ classes on the $\kappa$-legs
 with an equivalent vertex factor involving $\kappa$ classes. The equivalence
requires the following well-known equality. If
$$f  = \sum_{m \geq 1} c_m z^m$$ is a power series without constant term and 
$F = \exp(f)$ is the exponential, then
\begin{equation} \label{Eq:logexp}
\sum_{k \geq 0} p_{k*} \left[ \prod_{i=n+1}^{n+k} \psi_i (1-F(\psi_i))\right]
= \exp\left(-\sum_{m\geq 1} c_{m} \kappa_m\right),
\end{equation}
where $p_k :\oM_{g,n+k} \to \oM_{g,n}$ is the forgetful map.

As before, to every vertex $v$ of a stable graph we assign a formal variable $x_v$ satisfying $x_v^{r-1}=1$. For a polynomial $\Pi$ in variables $x_v$, we will denote by $\{ \Pi \}_x$ the term of degree~0 in all variables~$x_v$.

Given a stable graph,  to each edge $e$ we assign the {\em edge factor}{\footnote{Division
 by $\psi'+\psi''$ is only possible for the constant term in $r$ of the edge factor.}}
$$
\widetilde{\Delta}(e) = \sum_{a+b=r-2}\frac1{x^a y^b} \frac{1- \exp\left[- \sum\limits_{m \geq 1} 
\frac{B_{m+1}(a+1) \cdot (x\psi')^m + B_{m+1}(b+1) \cdot (y\psi'')^m }{m(m+1)} 
\right]}{\psi' + \psi''}.
$$
Here, $a$ and $b$ are non-negative integers, $\psi'$ and $\psi''$ are
the cotangent line classes corresponding to the half-edges of $e$, 
and $x,y$ are the vertex variables corresponding to the vertices adjacent to the edge (if the edge is a loop, then $x=y$).

Furthermore, to each leg $i$ we assign the {\em leg factor}
$$
L(i) = \frac1{x_v^{a_i}} \exp\left(- \sum_{m \geq 1} \frac{B_{m+1}(a_i+1)}{m(m+1)}  (x_v\psi_i)^m\right),
$$
where $\psi_i$ is the cotangent line class associated to the leg,
and $x_v$ is the vertex variable of the vertex
to which the leg is attached.

Finally, we assign to each vertex $v$ the {\em vertex factor}
$$
 \kappa(v) = x_v^{g-1} \exp \left[\sum_{m \geq 1} \frac{B_{m+1}(1)}{m(m+1)} x_v^m \kappa_m\right].
$$
Here, $x_v$ is the variable of the vertex and $\kappa_m$ is the $\kappa$ class on the moduli space $\oM_v$ corresponding to~$v$.

\begin{proposition} \label{Prop:Hol}
Let $a_1, \dots a_n\in \mathbb{Z}_{\geq 0}$ satisfy  $$\sum_{i=1}^n a_n=2g-2\, .$$
 The coefficient 
$(-1)^g \mathcal{W}_{g; a_1, \dots, a_n}(0)$ of the rescaled Witten 
class $$(-1)^g r^{g-1} \cdot W^r_g(a_1, \dots, a_n)$$ is given by 
the $r^0$ coefficient of 
the degree $g-1$ part of
$$
\sum_{\Gamma\in \mathsf{G}_{g,n}} \frac{(-1)^{g-1+h^1(\Gamma)}}{|\Aut(\Gamma)|} p_*
\left\{ \prod_e \widetilde{\Delta}(e)
\prod_{i=1}^n L(i) \prod_v \kappa(v)  \right\}_x,
$$
where $p: \oM_\Gamma \to \oM_{g,n}$ is the natural boundary map.
\end{proposition}

\begin{proof}
The result is obtained by a mod $r$ reduction of the formula
 of Proposition~\ref{Prop:SecondShift} for the rescaled Witten class. 
The polynomials $P_m(r,a)$ are replaced by the expression for $$Q_m(a) = P_m(0,a)$$ 
given in Proposition~\ref{Prop:Pcstterm}. 
The sum over weightings $\mathbf{a}$ has been moved to
the sum in the definition of the edge term $\widetilde{\Delta}(e)$.
The push-forward of the $\kappa$-factors associated to
 the legs marked by $n+1 \leq i \leq n+k$ adjacent to a vertex~$v$ is collected in the vertex factor $\kappa(v)$. 
The factor $r-1$ becomes $-1$. Finally, we multiply the resulting expression by the global factor of $(-1)^g$.
\end{proof}

The advantage of Proposition \ref{Prop:Hol}
is that all the inputs are explicit. Proposition \ref{Prop:SecondShift} is 
more general (and determines the full $r$
dependence of $\mathcal{W}_{g; a_1, \dots, a_n}(r)$), but
involves the polynomials $P(r,a )$ for which we know 
no closed formula.

\begin{remark}
The class $[\mathcal{H}_g(a_1,\ldots,a_n)]\in H^{2(g-1)}(\cM_{g,n})$ on the
moduli of nonsingular curves is easily determined by
the classical Thom-Porteous formula{\footnote{See, for example, 
\cite[Section 2]{Chen} for
the Thom-Porteous approach.}} to be
\begin{equation}\label{tptp}
c_{g-1}\left( 
R^1 \pi_*\omega\left(-\sum_{i=1}^n a_i \mathsf{p}_i\right) - R^0\pi_* \omega\left(-\sum_{i=1}^n a_i \mathsf{p}_i\right)\right)\, ,
\end{equation}
where $\pi:\mathcal{C} \rightarrow \cM_{g,n}$ is the universal curve.
We have checked that the restriction of 
Proposition~\ref{Prop:Hol} to $\cM_{g,n}$ agrees with
\eqref{tptp} calculated by
Grothendieck-Riemann-Roch. As a result, the restriction of Conjecture A.1
to $\cM_{g,n}$ is correct.

\end{remark}

\end{appendix}

\vspace{+16 pt}
\noindent Departement Mathematik, ETH Z\"urich\\ 
\noindent rahul@math.ethz.ch

\vspace{+8 pt}
\noindent
Department of Mathematics, MIT\\
apixton@mit.edu

\vspace{+8 pt}
\noindent
CNRS, Institut Math\'ematique de Jussieu\\
dimitri.zvonkine@ump-prg.fr


\begin{thebibliography}{99}

\bibitem{BCGGM} 
M. Bainbridge, D. Chen, Q. Gendron, S. Grushevsky, and  M. Moeller, {\em Compactification of strata of abelian differentials},
\texttt{arXiv:1604.08834}.

\bibitem{bp} P. Belorousski and R. Pandharipande, {\em A descendent relation in genus 2}, Ann. Scuola Norm. Sup. Pisa Cl. Sci. {\bf 29} (2000),
171-191.

\bibitem{Chen} D. Chen, {\em Strata of abelian differentials and the Teichmüller dynamics},
J. Mod. Dyn. {\bf 7} (2013), no.~1, 135--152. 

\bibitem{Chiodo} A. Chiodo, {\em The Witten top Chern class via K-theory}, 
J. Algebraic Geom. {\bf 15} (2006), 681--707.




\bibitem{Faber}
C.~Faber, 
\newblock{\em A conjectural description of the tautological ring of the moduli 
space of curves}, Moduli of curves and abelian varieties,  109--129, Aspects Math., 
Vieweg, Braunschweig, 1999.

\bibitem{FP1} C. Faber and R. Pandharipande (with an appendix
by D. Zagier), {\em Logarithmic series
and Hodge integrals in the tautological ring}, Michigan Math. J. {\bf 48}
(2000), 215--252.

\bibitem{FP3} C. Faber and R. Pandharipande, {\em Tautological and non-tautological cohomology of the moduli space of curves} in {\em Handbook of moduli, Vol. I}, 
293--330, Adv. Lect. Math. {\bf 24}, Int. Press: Somerville, MA, 2013.

\bibitem{FarP}
G. Farkas and R. Pandharipande, {\em The moduli space of twisted
canonical divisors}, J. Institute Math. Jussieu {\bf 17} (2018), 615--672.

\bibitem{Fulton} W. Fulton, {\em Intersection theory}, Springer-Verlag:
Berlin, 1984.
 
\bibitem{FH} W. Fulton and J. Harris, {\em Representation theory},
Springer-Verlag: Berlin, 1991.


\bibitem{Givental} A. Givental, {\em Gromov-Witten invariants and quantization of quadratic Hamiltonians},
 Mosc. Math. J. {\bf 1} (2001), 551--568.


\bibitem{GraPan} T. Graber and R. Pandharipande, {\em Constructions of nontautological classes on moduli spaces 
of curves}, Michigan Math. J. {\bf 51} (2003), 93--109.

\bibitem{fjr} H. Fan, T. Jarvis, and Y. Ruan, {\em The Witten equation,
mirror symmetry and quantum singularity theory}, Ann. of Math. (2) {\bf 178} (2013), no.~1, 1--106.








\bibitem{Janda14} F. Janda, {\em Comparing tautological relations
from the equivariant Gromov-Witten theory of projective spaces
and spin structures}, arXiv:1407.4778.

\bibitem{Janda15} F. Janda, {\em Frobenius manifolds near the discriminant
    and relations in the tautological ring}, Lett. Math. Phys {\bf 108} (2018),
  1649--1675.

\bibitem{Janda15*} F. Janda, {\em Relations on $\oM_{g,n}$ via
equivariant Gromov-Witten theory of $\mathbb{P}^1$}, Algebr. Geom. {\bf 4} (2017), 311--336.

\bibitem{JPPZ} F. Janda, R. Pandharipande, A. Pixton, D. Zvonkine. {\em 
    Double ramification cycles on the moduli spaces of curves.}, Publ. Math.
  IHES {\bf 125} (2017), 221--266.

\bibitem{KonMan} M. Kontsevich and Yu. Manin, {\em Gromov-Witten classes,
quantum cohomology, and enumerative geometry}, Comm. Math. Phys. {\bf 164} (1994), 525--562.

\bibitem{Looij} E. Looijenga, {\em On the tautological ring of $M_g$}.  
Invent. Math. {\bf 121}  (1995), 411--419.

\bibitem{Mochizuki} T. Mochizuki,
{\em The virtual class of the moduli stack of stable r-spin curves},
Comm. Math. Phys. {\bf 264} (2006), no.~1, 1--40.

\bibitem{Mum} D. Mumford, {\em Towards an enumerative geometry of
the moduli space of curves} in {\em Arithmetic and Geometry}
(M. Artin and J. Tate, eds.), Part II, Birkh\"auser, 1983, 271--328.

\bibitem{kap} R. Pandharipande, {\em The $\kappa$ ring of the
moduli of curves of compact type}, Acta Math. {\bf 208} (2012), 335--388.

\bibitem{SLC} R. Pandharipande, {\em A calculus for the moduli space of curves}, Proceedings of Algebraic geometry -- Salt Lake City 2015, Proc. Sympos. Pure
  Math. {\bf 97}, Part 1, 459--488.

\bibitem{PPP} R. Pandharipande, {\em Cohomological field theory
    calculations}, Proceedings of the ICM (Rio de Janeiro 2018),
Vol. 1, 869--898.

  


\bibitem{PaPiZv} R. Pandharipande, A. Pixton, and D. Zvonkine, {\em Relations on $\oM_{g,n}$ via 3-spin structures}, J. Amer. Math. Soc. {\bf 28} (2015), 279--309.

\bibitem{Pixton} A. Pixton, 
{\em Conjectural relations in the tautological ring of $\oM_{g,n}$}, \texttt{arXiv:1207.1918}.

\bibitem{PiThesis} A. Pixton,
{\em The tautological ring of the moduli space of curves}, Princeton Ph.D. 2013.

\bibitem{PolVai} A. Polishchuk and A. Vaintrob, {\em Algebraic construction of 
Witten's top Chern class} in {\em Advances in algebraic geometry motivated by physics (Lowell, MA, 2000)}, 229-249, Contemp. Math. {\bf 276}, AMS: Providence, RI, 2001.

\bibitem{Polishchuk}  A. Polishchuk, {\em Witten's top Chern class on the moduli space of higher spin curves} in
 {\em Frobenius manifolds}, 253-264, Aspects Math. {\bf E36}, Vieweg: Wiesbaden, 2004. 

\bibitem{Pol} A. Polishchuk, {\em{Moduli spaces of curves with effective $r$-spin structures}}  in  {\em Gromov-Witten theory of spin curves and orbifolds},
1–20,
Contemporary  Mathematics {\bf 403},  AMS: Providence, RI, 2006.


\bibitem{Sauv} A.~Sauvaget, {\em Cohomology classes of strata of differentials}, arXiv:1701.07867.


\bibitem{Teleman} C. Teleman, {\em The structure of 2D semi-simple field theories},
 Invent. Math. {\bf 188} (2012), 525--588.

\bibitem{Verlinde} E. Verlinde, {\em Fusion rules and modular transformations in 2D conformal field theory,}
Nuclear Physics B300 [FS22] (1988), 360--376.

\bibitem{Witten}  E. Witten, {\em Algebraic geometry associated with matrix models of two-dimensional gravity}
in {\em Topological methods in modern mathematics (Stony Brook, NY, 1991)}, 
235--269, Publish or Perish: Houston, TX, 1993.

\end{thebibliography}
\end{document}